\documentclass[a4paper,11pt]{amsart}
\usepackage[utf8]{inputenc}
\usepackage[a4paper]{geometry} 
\usepackage{mathtools,amssymb,bm,amsthm} 
\usepackage{amsfonts}
\usepackage[utf8]{inputenc}
\usepackage[english]{babel} 
\usepackage{xcolor} 
\usepackage{graphicx} 
\usepackage[pdfencoding=auto]{hyperref} 
\usepackage{stmaryrd, enumitem}
\usepackage{dsfont}

\newtheorem{theorem}{Theorem}[section]
\newtheorem{proposition}[theorem]{Proposition}
\newtheorem{corollary}[theorem]{Corollary}
\newtheorem{lemma}[theorem]{Lemma}
\newtheorem{fact}[theorem]{Fact}
\newtheorem*{step1}{Step 1}
\newtheorem*{step2}{Step 2}
\newtheorem*{step3}{Step 3}
\newtheorem*{step4}{Step 4}
\newtheorem*{step5}{Step 5}
\newtheorem*{step6}{Step 6}
\newtheorem*{step7}{Step 7}
\newtheorem*{step8}{Step 8}
\theoremstyle{remark}
\newtheorem*{remark}{Remark}
\newtheorem*{rem1}{Remark 1}
\newtheorem*{rem2}{Remark 2}
\newtheorem*{rem3}{Remark 3}

\theoremstyle{definition}

\newtheorem{example}[theorem]{Example}
\numberwithin{equation}{section}
\newcommand{\R}{\mathbb{R}}
\newcommand{\N}{\mathbb{N}}
\newcommand{\K}{\mathbb{K}}
\newcommand{\Z}{\mathbb{Z}}

\newcommand{\F}{\mathcal{F}}
\newcommand{\Sc}{\mathcal{S}}

\newcommand{\A}{\mathcal{A}}

\newcommand{\minus}{\setminus}

 \makeatletter
\@namedef{subjclassname@2020}{%
  \textup{2020} Mathematics Subject Classification}
\makeatother

\author[M. Fakhoury]{Micheline Fakhoury}
\address{Univ. Artois, UR 2462, Laboratoire de Math\'{e}matiques de Lens (LML)\\ F-62300 Lens, France}
\email{micheline.fakhoury@univ-artois.fr}

\begin{document}
\title[combinatorial Banach spaces]{Isometries of $p\,$-$\,$convexified \\combinatorial Banach spaces}

\begin{abstract} We show that if $1<p\neq 2<\infty$, then any isometry of the $p\,$-$\,$convexification of the combinatorial Banach space associated with a hereditary family of finite subsets of $\N$ containing the singletons is given by a signed permutation of the canonical basis. In the case of a generalized Schreier family, the result also holds for $p=2$, and every isometry is diagonal. These results are deduced from more general theorems concerning combinatorial-like Banach spaces.
\end{abstract}

\thanks{The author would like to thank her advisor \'Etienne Matheron for his help in the writing of the paper.}
\thanks{This work was supported in part by
the project FRONT of the French
National Research Agency (grant ANR-17-CE40-0021)}

\maketitle
\section{Introduction}
In what follows, the word ``isometry'' will always mean ``linear surjective isometry''. The study of isometries in Banach spaces has a very long and still ongoing history, see e.g. \cite{FJ} or the recent survey \cite{AB}. In the present paper, we will focus on a class of spaces usually called ``{combinatorial Banach spaces}'', that are defined in terms of certain families of finite subsets of $\N$. The isometries of combinatorial Banach spaces  have been analyzed recently in \cite{ABC}, \cite{BFT} and \cite{BP}. In this paper, we look at the {$p\,$-$\,$convexifications} of such spaces, $1<p<\infty$. 

\medskip
Throughout the paper, the vector spaces are real or complex, and the scalar field is denoted by $\mathbb K$. We denote by $c_{00}$ the vector space of all finitely supported scalar-valued sequences, and by $(e_i)_{i\in\N}$ the canonical basis of $c_{00}$. Also, we point out that $\N$ starts at~$1$.

\medskip
By a \emph{combinatorial family}, we mean a family of finite subsets of $\N$ that contains all singletons and is hereditary for inclusion: if $F \in \F$ and $G \subset F$ then $G \in \F$. We say that $\F$ is \emph{compact} if it is compact when considered as a subset of $\{ 0, 1\}^\N$ \textit{via} the identification of a set $F\subset\N$ with its characteristic function; in other words, if $\F$ is closed in $\{ 0,1\}^\N$. By the hereditarity property, $\F$ is compact if and only if its closure $\overline{\F\,}$ in $\{ 0,1\}^\N$ contains only finite sets. More generally, a set $A\subset \N$ is in $\overline{\F\,}$ if and only if every finite subset of $A$ is in $\F$.

\smallskip
Well known examples of compact combinatorial families are the following: for every natural number $n$, the family $\F_{\leqslant  n}:=\{ F\subset\N;\; \vert F\vert\leqslant  n\}$; the classical  \emph{Schreier family} $ \mathcal S_1:=\{\emptyset\}\cup\{ \emptyset\neq F\subset\N;\; \vert F\vert\leqslant \min(F)\};$
and the generalized Schreier families $\mathcal S_\alpha$, $\alpha<\omega_1$ introduced in \cite{AA} (see Section \ref{Schreier}). As for non-compact families, the most obvious example is  $\F:= {\rm FIN}$, the family of all finite subsets of $\N$.

\medskip 
Given a combinatorial family $\F$, the \emph{combinatorial Banach space} $X_\F$ is defined to be the completion of $c_{00}$  with respect to the following norm: 
\[ \Vert x \Vert_\F := \sup \biggl\{\sum_{i\in F} \vert x(i)\vert; \; F \in \F\biggr\}. \]
Note that the only property of $\F$ that is needed for this formula to define a norm is that $\mathcal F$ contains all singletons. The two ``extreme'' examples are $\F:= \F_{\leqslant  1}$, for which $X_\F=c_0$, and $\F:= {\rm FIN}$, for which $X_\F=\ell_1$. As an ``intermediate'' example, for the family 
$\F:=\{ F\in{\rm FIN};\; F\subset 2\N\;{\rm or}\; F=\{ n\}\;\hbox{for some $n\in 2\N-1$}\}$ one gets the $\ell_\infty\,$-$\,$direct sum $\ell_1(2\N)\oplus c_0(2\N-1)$.

\smallskip
It is clear by definition that $(e_i)_{i\in\N}$ is a $1$-unconditional Schauder basis for $X_\F$. Moreover, it is known that if $\F$ is compact then $(e_i)$ is shrinking, \textit{i.e.} the associated sequence of coordinate functionals 
$(e_i^*)_{i\in\N}$ is a Schauder basis of $X_{\F}^*$, but not boundedly complete. This follows from the fact that $X_\mathcal F$ is isomorphic to a subspace of $\mathcal C(\F)$ (see e.g. \cite{BiLau}), and hence $X_{\mathcal F}$ is $c_0\,$-$\,$saturated because $\F$ is a countable 
compact metric space (\cite{BessPel}, see also \cite{Ro}). As such, $X_\F$ is non-reflexive and does not contain $\ell_1$, which gives the result since the basis $(e_i)$ is unconditional (\cite{J}, see also \cite{LT} or \cite{AK}).

\medskip
The space $X_{\mathcal F}$ has some immediately apparent isometries: if $\pi:\N\to \N$ is a permutation of $\N$ preserving 
the family $\mathcal F$ (\textit{i.e.} $\pi(\mathcal F)=\mathcal F$, where $\pi(\mathcal F):=\{ \pi(F);\; F\in\mathcal F\}$) and if $(\omega_i)_{i\in\N}$ is any sequence of signs (\mbox{\it i.e.} scalars of modulus $1$), then the formula $Je_i:=\omega_i e_{\pi(i)}$, $i\in\N$ uniquely defines an isometry of 
$X_{\mathcal F}$. Such an isometry is called an \emph{$\F$-$\,$compatible signed permutation} of the basis $(e_i)$. If $\F$ is compact then, since $(e_i^*)$ is a basis of $X_\F^*$, one can define in the same way $\F$-$\,$compatible signed permutations of $(e_i^*)$, which are isometries of $X_{\mathcal F}^*$ (more precisely, the adjoints of the $\F$-$\,$compatible signed permutations of $(e_i)$). It is quite natural to wonder whether any isometry of $X_\F$ or $X_\F^*$ arises in that way. This question is answered positively in \cite{ABC} for Schreier families of finite order, and in \cite{BFT} and \cite{BP} for more general families. In particular, the following results have been obtained.

\begin{itemize}
\item[(i)] Assume that the combinatorial family $\mathcal F$ is compact and \emph{spreading}, which means that for any $F \in \F$ and $\sigma : F \to \N$ such that  $\sigma(n) \geqslant   n$ for all $n \in F$, 
we have $\sigma(F) \in \F$. Then, any isometry of $X_{\F}^*$ is an $\F$-$\,$compatible signed permutation of $(e_i^*)$, and hence any isometry of $X_\F$ is an $\F$-$\,$compatible signed permutation of $(e_i)$. (\cite{BFT})
\item[(ii)] Assume only that $\F$ is compact and that the closure of $\F^{MAX}$ in $\{ 0,1\}^\N$ contains all singletons, where $\F^{MAX}$ is the set of all maximal elements of $\F$. Then, any isometry of $X_\F$ is an $\F$-$\,$compatible signed permutation of $(e_i)$. (\cite{BP})
\item[(iii)] More generally, let $\F_1$ and $\F_2$ be two compact combinatorial families such that $\overline{\F_1^{MAX}}$ and $\overline{\F_2^{MAX}}$ contain all singletons. Then, $X_{\F_1}$ and $X_{\F_2}$ are isometric if and only if there exists a permutation $\pi$ of $\N$ such that $\pi(\F_1)=\F_2$; and in that case any isometry between $X_{\F_1}$ and $X_{\F_2}$ is an $(\F_1,\F_2)\,$-$\,$compatible signed permutation of $(e_i)$. Moreover, if $\F_1$ and $\F_2$ are both spreading, then $X_{\F_1}$ and $X_{\F_2}$ can be  isometric only when $\F_1=\F_2$. (\cite{BP})
\item[(iv)] If $\mathcal F$ is a generalized Schreier family $\mathcal S_\alpha$, then any isometry of $X_\F^*$ is \emph{diagonal}, \textit{i.e.} it has the form $Je_i^*=\omega_i e_i^*$ for some sequence of signs 
$(\omega_i)$. (\cite{ABC}, \cite{BFT})
\end{itemize}

Note that (i) implies in particular that if $\F$ is compact and spreading, then any isometry of $X_\F^*$ is $w^*$-$\,w^*$ continuous, being the adjoint of an isometry of $X_\F$. We will see below that when $\K=\mathbb C$, the spreading assumption can be dispensed with.

\medskip
As already mentioned, in the present paper we  look at \emph{$p\,$-$\,$convexifications} of combinatorial Banach spaces. 
Given a combinatorial family $\F$ and $p\in (1,\infty)$, the $p\,$-$\,$convexification of $X_\F$, denoted by $X_{\F,p}$, is defined to be the completion of $c_{00}$ with respect to the norm 
\[ \Vert x \Vert_{\F,p} := \sup \biggl\{\biggl(\sum_{i\in F} \vert x(i)\vert^p\biggr)^{{1}/{p}}; \; F \in \F\biggr\}. \]

Again, $(e_i)_{i\in\N}$ is a $1$-unconditional basis of $X_{\mathcal F, p}$. It is always shrinking, but not boundedly complete if $\F$ is compact (see \cite{BiLau} or Section \ref{blabla}). Moreover, $\F\,$-$\,$compatible signed permutations of $(e_i)$ and $(e_i^*)$ are still well-defined isometries, and the latter are the adjoints of the former. We are going to prove the following theorem.
\begin{theorem}\label{main} Let $\F$ be a combinatorial family, and let $1<p<\infty$.
\begin{enumerate}
\item[\rm (1)] If $p\neq 2$, then any isometry of $X_{\mathcal F,p}^*$ is an $\F$-$\,$compatible signed permutation of $(e_i^*)$, and hence any isometry of $X_{\mathcal F, p}$ is an $\F$-$\,$compatible signed permutation of $(e_i)$.
\item[\rm (2)] Assume that $\mathcal F$ is a generalized Schreier family $\mathcal S_\alpha$, $1\leqslant\alpha<\omega_1$. Then: if $p\neq 2$, any isometry of $X_{\mathcal F, p}^*$ is diagonal; and if $p=2$, any isometry of $X_{\mathcal F, 2}$ is diagonal.
\end{enumerate}
\end{theorem}

\smallskip The main difference between this theorem and the above results (i), (ii) from \cite{BFT} and~\cite{BP} is that in (1), we do not need to make any additional assumption on the family $\mathcal F$ (not even compactness). On the other hand, when $p=2$ we are not able to get the formally stronger conclusion concerning isometries of $X_{\F,2}^*$ in (2). This is most likely due to a deficiency in our proof: if $p\neq 2$, we deduce the result from (1) and the 
general fact (hinted at in \cite{BP}) that the only permutation $\pi:\N\to\N$ preserving $\mathcal S_\alpha$ is the identity, whereas if $p\neq 2$, we adapt  arguments from \cite{ABC} that do not seem to ``dualize'' in an obvious way. (However, in the complex case we do get the result for $X_{\Sc_\alpha,2}^*$; see Corollary~\ref{Scfin}.)  Moreover, (2) is stated in a somewhat vague way: see Proposition \ref{restated} for the precise statement.

\smallskip Let us also note that in (1), the assumption that $p\neq 2$ is clearly necessary to get the result for every combinatorial family (take $\F:={\rm FIN}$, for which $X_{\F,2}=\ell_2$). It is still necessary if we restrict ourselves to compact families $\mathcal F$, as shown by the following example.

\begin{example}\label{counter}
Let $\F:= \{F \subset \N; \; F \subseteq \{1,2\} \text{ or } F=\{n\} \text{ for some } n\geqslant   3\}$ and let $J: X_{\F,2}\to X_{\F,2}$ be the operator defined as follows: 
\[
\begin{aligned}
   Je_i:= \begin{cases}
    \;\;\,\cos{\alpha}\, e_1+ \sin{\alpha}\, e_2&\text{if }i=1,\\
   -\sin{\alpha}\, e_1+ \cos{\alpha}\, e_2  &\text{if }i=2,\\
   \;\;\;  e_i &\text{otherwise},
  \end{cases} 
\end{aligned}\]
where $\alpha\in \R\setminus (\pi\mathbb Z \cup \frac{\pi}{2}\Z)$. 
Then $J$ is an isometry which is not a signed permutation of the basis $(e_i)$. The same example works with the family 
\[ \mathcal F:=\bigl\{ F\subset\N;\; F\subseteq \{ 1,2\}\;{\rm or}\; F\subseteq I_n\;\hbox{for some $n$}\bigr\},\]
where $(I_n)_{n\in\N}$ is any partition of $\N\setminus\{ 1,2\}$ into consecutive intervals. This shows that when $p= 2$, part (1) of Theorem \ref{main} can fail even for compact families $\mathcal F$ containing sets of arbitrarily large cardinality.
\end{example}

\smallskip To prove part (1) of Theorem \ref{main}, we will follow closely the strategy used in 
\cite{BFT} and~\cite{BP} to get the corresponding result for $X_\F^*$. In particular, a key step is the description of the extreme points of the unit ball of $X_{\F,p}^*$. However, in the complex case $(\K=\mathbb C)$ one can give a quite different proof by making use of the theory developed by N. J. Kalton and G. V. Wood in \cite{KW}. This approach also allows to prove  the following proposition  concerning $X_\F$ and $X_{\mathcal F}^*$. This generalizes the main result of \cite{BFT}, but in the complex case only. That we need $\K=\mathbb C$  is not accidental: it is shown in \cite{AC} that the result is \emph{not} true when $\K=\R$. 

\begin{proposition}\label{l1} Assume that $\K=\mathbb C$. Let $\mathcal F$ be a combinatorial family. 
\begin{enumerate}
\item[\rm (a)] Any isometry of $X_{\F}$ is an $\F$-$\,$compatible signed permutation of $(e_i)$.
\item[\rm (b)] If $\F$ is compact, then any isometry of $X_{\F}^*$ is an $\F$-$\,$compatible signed permutation of $(e_i^*)$.
\end{enumerate} 
\end{proposition}

\smallskip The rest of the paper is organized as follows. In Section \ref{blabla}, we prove part (1) of Theorem \ref{main} and Proposition \ref{l1}. We obtain in fact more general results concerning Banach spaces of the form $X_{\F,\mathbf E}$, where $\F$ is a combinatorial family and $\mathbf E$ is a Banach space endowed with a normalized $1$-unconditional basis (Theorem \ref{nonsense}, Corollary \ref{l1bis} and Theorem \ref{nonsensebis}).  We also give some examples involving Orlicz sequence spaces to illustrate these results. 
In Section \ref{Schreier}, we recall the definition of the Schreier families $\mathcal S_\alpha$, prove a few facts concerning these families, and prove part (2) of Theorem \ref{main}. As mentioned above, the approach will be different when $p\neq 2$ and when $p=2$. Finally, in Section \ref{more} we use \cite{KW} to obtain, in the complex case, a characterization of those families  $\F$ such that all isometries of $X_{\F,2}$ are signed permutations of $(e_i)$ (Proposition \ref{carac}). We point out that this characterization is due to C. Brech and A. Tcaciuc  \cite{AC}, and that it is valid also in the real case. The ``novelty'' in the present paper is to explain how the methods of \cite{KW} naturally lead to the result. We would like to thank C. Brech  and A. Tcaciuc for showing us their result, and for allowing us to use some of their ideas.

\section{Arbitrary families and $p\neq 2$}\label{blabla}

\subsection{An abstract result} In this section, we prove part (1) of Theorem \ref{main}. Since this requires not much additional effort, we do this in a more general context, which we now describe.

\medskip We fix a Banach space $(\mathbf E, \Vert\,\cdot\,\Vert_{\mathbf E})$ admitting a normalized, $1$-unconditional basis $(\mathbf e_i)_{i\in\N}$, and we denote by $(\mathbf e_i^*)$ the associated sequence of coordinate functionals. 

\medskip Recall that $(e_i)_{i\in\N}$ is the canonical basis of $c_{00}$. For any $x=\sum_{i\in\N} x(i) e_i \in c_{00}$, we denote by $\mathbf x$ the ``copy'' of $x$ living in $\mathbf E$, \mbox{\it i.e.} \[ \mathbf x:=\sum_{i\in\N} x(i) \mathbf e_i;\] and if $\F$ is a combinatorial family, we set
\[ \Vert x\Vert_{\mathcal F, \mathbf E}:= \sup\,\left\{ \Bigl\Vert \sum_{i\in F} x(i) \mathbf e_i\Bigr\Vert_{\mathbf E};\; F\in\mathcal F\right\} =\sup\,\bigl\{ \Vert P_F(\mathbf x)\Vert_{\mathbf E};\; F\in \mathcal F\bigr\},\]
where the notation $P_F$ should be self-explanatory.

\smallskip  We denote by $X_{\mathcal F,\mathbf E}$ the completion of $( c_{00}, \Vert \,\cdot\,\Vert_{\mathcal F,\mathbf E})$; and for notational simplicity, we will often write $\Vert\,\cdot\,\Vert$ instead of $\Vert \,\cdot\,\Vert_{\mathcal F,\mathbf E}$. 

\medskip By definition, it is clear that $(e_i)_{i\in\N}$ is a $1$-unconditional basis of $X_{\mathcal F,\mathbf E}$: indeed, $(e_i)$ is a $1$-unconditional sequence because $(\mathbf e_i)$ is, and $\overline{\rm span}\, ( e_i;\; i\in\N)=X_{\mathcal F, \mathbf E}$. The basis $(e_i)$ may not be shrinking, so we set \[ Z_{\F,\mathbf E}:=\overline{{\rm span}(e_i^*,\, i\in\N)}^{X_{\F,\mathbf E}^*}.\]

 We are going to prove the following theorem.
 
 \begin{theorem}\label{nonsense} Assume that the dual space $\mathbf E^*$ is strictly convex and that the following assumption holds true: 
 \begin{itemize}
 \item[ {\rm (H)}] For any $j\neq k$ in $\N$, if $T:{\rm span}(\mathbf e_j^*,\mathbf e_k^*)\to \mathbf E^*$ is a linear isometric embedding, then the vectors $T\mathbf e_j^*$ and $T\mathbf e_k^*$ have disjoint supports.
 \end{itemize}

Let $\F_1$ and $\F_2$ be two combinatorial families, and let $J: X_{\F_1,\mathbf E}^*\to X_{\F_2,\mathbf E}^*$ be an isometry. If either $J$ is $w^*$-$\, w^*$ continuous, or  $J(Z_{\F_1,\mathbf E})=Z_{\F_2,\mathbf E}$,  then there exist a permutation $\pi:\N\to\N$ with $\pi(\F_1)=\F_2$ and a sequence of signs 
$(\omega_i)$ such that $Je_i^*=\omega_i e_{\pi(i)}^*$ for all $i\in\N$. The second alternative occurs as soon as $\F_1$ and $\F_2$ are compact.

\smallskip Moreover, when $\mathbb K=\mathbb C$ one can replace {\rm (H)} by the assumption that ${\rm span}(\mathbf e_j^*,\mathbf e_k^*)$ is not isometrically equal to $\ell_2(2)$ for any $j\neq k$ in $\N$.

\smallskip Finally, if the families $\F_1$ and $\F_2$ are compact, then the strict convexity of $\mathbf E^*$ may be replaced by the  assumption that $\Vert\,\cdot\,\Vert_{\bf E}$ is smooth at every finitely supported vector $\mathbf x\in S_{\bf E}$, the unit sphere of $\bf E$; and it is enough to know that {\rm (H)} holds true for isometric embeddings $T:{\rm span}(\mathbf e_j^*,\mathbf e_k^*)\to {\rm span}(\mathbf e_i^*,\,i\in\N)$.

\end{theorem}

\medskip The unexplained notation and terminology are as follows. 
If $Z$ is a Banach space with a Schauder basis $(f_i)_{i\in\N}$, we define the \emph{support} of a vector $z^*\in Z^*$ to be the set $\{ i\in\N;\; \langle z^*, f_i\rangle\neq 0\}$; which makes sense even if $(f_i^*)$ is not a basis of $Z^*$. 
If $u$ and $v$ are two vectors in a normed space $Z$, we say that \emph{${\rm span}(u,v)$ is isometrically equal to $\ell_2(2)$} if $\Vert au+bv\Vert^2=\vert a\vert^2+\vert b\vert^2$ for any $a,b\in\K$. Finally, $\Vert\,\cdot\,\Vert_{\bf E}$ is said to be \emph{smooth} at some vector $\mathbf x\in S_{\bf E}$ if there is exactly one linear functional $\mathbf x^*\in S_{\mathbf E^*}$ such that $\langle \mathbf x^*, \mathbf x\rangle=1$.

\medskip It is well known, and easy to see from any proof that the isometries of $\ell_q$, $q\neq 2$ are signed permutations of the canonical basis (see e.g. \cite{Car} or \cite{AB}), that assumption (H) is satisfied if $\mathbf E=\ell_p$, $1<p<\infty$ and $p\neq 2$. Moreover, we already mentioned that 
$(e_i)$ is a shrinking basis of $X_{\F, p}$ for any combinatorial family $\F$, \mbox{\it i.e.} $Z_{\F,p}=X_{\F,p}^*$. So Theorem~\ref{nonsense} immediately implies part (1) of Theorem \ref{main}. More generally, Theorem \ref{nonsense} applies to a reasonably large class of Orlicz sequence spaces, see Section \ref{Exs}.

\smallskip Let us also note that if the basis $(\mathbf e_i)$ of $\mathbf E$ is \emph{$1$-symmetric}, which means that all permutations $\pi:\N\to\N$ induce well-defined isometries of $\mathbf E$, then any  permutation $\pi$ such that $\pi(\mathcal F_1)=\mathcal F_2$ does induce an isometry  between  $Z_{\F_1,\mathbf E}$ and $Z_{\F_2,\mathbf E}$.
\begin{rem1} 
It is well known that if the basis $(\mathbf e_i)$ is $1$-symmetric and $\mathbf E$ is not isometric to a Hilbert space, then any isometry of $\mathbf E$ is given by a signed permutation of $(\mathbf e_i)$ (\cite{russian} in the real case, \cite{T} in the complex case). So it  is tempting to believe that 
if $\mathbf E^*$ is strictly convex, $(\mathbf e_i)$ is $1$-symmetric and $\mathbf E$ is not a Hilbert space then, for any combinatorial families $\F_1$, $\F_2$,  every isometry $J: X_{\F_1,\mathbf E}\to X_{\F_2,\mathbf E}$ is a signed permutation of $(e_i)$. This is however not true, see Section \ref{Exs}.
\end{rem1}

\begin{rem2} If one wants that for \emph{any} (compact) combinatorial family $\F$, all isometries of $X_{\F,\mathbf E}$ are signed permutation of $(e_i)$, then it is necessary to assume that ${\rm span}(\mathbf e_j,\mathbf e_k)$ is not isometrically equal to $\ell_2(2)$ for all $j\neq k$. Indeed, if ${\rm span}(\mathbf e_j,\mathbf e_k)=\ell_2(2)$ for some $j,k$, set $\F:=\{ F\in{\rm FIN};\; F\subseteq \{ j,k\}\;{\rm or}\; F=\{ n\}\;\hbox{for some $n\neq j,k$}\}$ and consider the same operator as in Example \ref{counter}, with $\{j,k\}$ in place of $\{ 1,2\}$. In fact, this example works for any compact family $\F$ such that $\{ j,k\}\in \F^{MAX}$ and every $G\neq \{ j,k\}$ in $ \F^{MAX}$ is disjoint from $\{ j,k\}$.
\end{rem2}

\medskip As in \cite{BFT} and \cite{BP}, Theorem \ref{main} (1) has the following consequence.
\begin{corollary} Let $\F_1$ and $\F_2$ be two combinatorial families, and let $1<p\neq 2<\infty$. Then, the following are equivalent: 
\begin{enumerate}
    \item[\rm (1)] $X_{\F_1,p}$ and $X_{\F_2,p}$ are isometric;
    \item[\rm (2)] $X_{\F_1,p}^*$ and $X_{\F_2,p}^*$ are isometric;
    \item[\rm (3)] there exists a permutation $\pi$ of $\N$ such that $\pi(\mathcal F_1)=\mathcal F_2$.
    \end{enumerate}
\end{corollary}
\begin{proof}
    (1) implies (2) for any two Banach spaces, and (2) implies (3) by Theorem \ref{nonsense} because $(e_i)$ is a shrinking basis of $X_{\F,p}$ for any combinatorial family $\F$. Finally, if $\pi$ is as in (3), it is easily seen that the formula $Je_i:= e_{\pi(i)}$ defines an isometry from $(c_{00}, \Vert\,\cdot\,\Vert_{\F_1,p})$ onto $(c_{00}, \Vert\,\cdot\,\Vert_{\F_2,p})$, which extends to an isometry $J: X_{\F_1,p}\to X_{\F_2,p}$.
    \end{proof}
    
\smallskip
The condition that there exists a permutation $\pi$ such that $\pi(\F_1)=\F_2$ is very restrictive. Indeed, it is shown in \cite{BP} that if this happens and if the combinatorial families $\F_1$ and $\F_2$ are compact and spreading, then in fact $\F_1=\F_2$. So we get
\begin{corollary}  Let the combinatorial families $\F_1$ and $\F_2$ be compact and spreading, and let $1<p\neq 2<\infty$. Then $X_{\F_1,p}^*$ and $X_{\F_2,p}^*$ are isometric {\rm (}if and{\rm )} only if $\F_1=\F_2$.
\end{corollary}
\smallskip In what follows, the symbols $\mathbf E$ and $\F$ always denote (respectively) a Banach space with a normalized $1$-unconditional basis $(\mathbf e_i)_{i\in\N}$ and a combinatorial family. We will first prove some simple general facts concerning $X_{\F,\mathbf E}$, and then prove Theorem \ref{nonsense}, a general version of Proposition \ref{l1}, and a variant of Theorem \ref{nonsense}.

\subsection{Norm inequalities and equalities} We first note that if $x\in c_{00}$ and if $\mathbf x=\sum_{i\in\N} x(i) \mathbf e_i$ is the copy of $x$ living in $\mathbf E$, then \[ \Vert x\Vert_{\mathcal F,\mathbf E}\leqslant \Vert \mathbf x\Vert_{\mathbf E},\]
by $1$-unconditionality of the basis $(\mathbf e_i)$. 
So there is a well defined ``canonical injection'' $\Lambda:\mathbf E\to X_{\mathcal F,\mathbf E}$ with $\Vert \Lambda\Vert\leqslant 1$: if $u\in \mathbf E$, then 
\[ \Lambda u=\sum_{i=1}^\infty \langle  \mathbf e_i^*,  u\rangle \, e_i,\]
where the series is $\Vert\,\cdot\,\Vert\,$-$\,$convergent in $X_{\mathcal F, \mathbf E}$. Consequently, to any functional $x^*\in X_{\mathcal F,\mathbf E}^*$, there is a canonically associated functional $\mathbf x^*\in \mathbf E^*$ such that $\Vert \mathbf x^*\Vert_{\mathbf E^*}\leqslant \Vert x^*\Vert$, namely $\mathbf x^*:=\Lambda^*x^*$. Explicitely,
\[ \mathbf x^*=\sum_{i=1}^\infty \langle x^*,e_i\rangle \, \mathbf e_i^*,\]
where the series is $w^*$-$\,$convergent in $\mathbf E^*$. The operator $\Lambda^*:x^*\mapsto \mathbf x^*$ is injective because $\Lambda$ has a dense range. The following simple facts will be essential for the proof of Theorem~\ref{nonsense}.

\begin{lemma}\label{superbasic} If $F\in\mathcal F$, then $\Vert x\Vert =\Vert \mathbf x\Vert_{\mathbf E}$ for all $x\in {\rm span}\,( e_i;\; i\in F)$ and $\Vert x^*\Vert =\Vert \mathbf x^*\Vert_{\mathbf E^*}$ for all $x^*\in{\rm span}\,( e_i^*; \; i\in F)$.
\end{lemma}
\begin{proof} The first part is clear by definition of $\Vert\,\cdot\,\Vert=\Vert\,\cdot\,\Vert_{\mathcal F, \mathbf E}$. If $x^*\in{\rm span}\,( e_i^*; \; i\in F)$ then, for all $x\in X_{\mathcal F,\mathbf E}$, we have 
\[  \langle x^*, x\rangle= \langle x^*, P_F(x)\rangle=\langle \mathbf x^*, P_F(\mathbf x)\rangle,\]
and hence $\vert  \langle x^*, x\rangle\vert\leqslant \Vert \mathbf x^*\Vert_{\mathbf E^*}\, \Vert P_F(\mathbf x)\Vert_{\bf E}\leqslant \Vert \mathbf x^*\Vert_{\mathbf E^*}\, \Vert x\Vert$. So $\Vert x^*\Vert \leqslant \Vert \mathbf x^*\Vert_{\mathbf E^*}$ for all $x^*\in {\rm span}\,( e_i^*; \; i\in F)$; and the reverse inequality is true for any $x^*\in X_{\mathcal F,\mathbf E}^*$.
\end{proof}

\begin{corollary}\label{corsuperbasic} If $A\in\overline{\F\,}$, then $\Vert \Lambda \mathbf x\Vert =\Vert \mathbf x\Vert_{\bf E}$ for every $\mathbf x\in\mathbf E$ supported on $A$, and $\Vert x^*\Vert=\Vert \mathbf x^*\Vert_{\mathbf E^*}$ for every $x^*\in X_{\F,\mathbf E}^*$ supported on $A$.
\end{corollary}
\begin{proof} Let $\mathbf x\in \mathbf E$ and $x^*\in X_{\F,\mathbf E}^*$ be supported on $A$. For $n\in\N$, set $F_n:= \llbracket 1,n\rrbracket\cap A$. Then $F_n\in \F$, so $\Vert \Lambda P_{F_n}\mathbf x\Vert=\Vert P_{F_n}\mathbf x\Vert_{\mathbf E}$ and $\Vert P_{F_n}x^*\Vert=\Vert P_{F_n}\mathbf x^*\Vert_{\mathbf E^*}$. Since $P_{F_n}\mathbf x\to \mathbf x$ in norm as $n\to\infty$, we immediately get $\Vert \Lambda \mathbf x\Vert =\Vert \mathbf x\Vert_{\bf E}$. Moreover, by $1$-unconditionality and since $P_{F_n}x^*\xrightarrow{w^*} x^*$, we have $\Vert x^*\Vert \leqslant \liminf \Vert P_{F_n} x^*\Vert \leqslant \limsup \Vert P_{F_n} x^*\Vert\leqslant \Vert x^*\Vert$, so $\Vert P_{F_n} x^*\Vert\to \Vert x^*\Vert$. Similarly $\Vert P_{F_n} \mathbf x^*\Vert_{\mathbf E^*}\to \Vert \mathbf x^*\Vert_{\mathbf E^*}$, so $\Vert x^*\Vert=\Vert \mathbf x^*\Vert_{\mathbf E^*}$.
\end{proof}

\subsection{Shrinking and boundedly complete}
In the next two facts, we give two general conditions ensuring that the basis $(e_i)$ of $X_{\F,\mathbf E}$ is shrinking, and a condition ensuring that it is not boundedly complete. 

\begin{fact}\label{domi} If either $(\mathbf e_i)$ is a shrinking basis of $\mathbf E$ or the family $\F$ is compact, then $(e_i)$ is a shrinking basis of $X_{\mathcal F,\mathbf E}$.
\end{fact}
\begin{proof} It is enough to show that any normalized block sequence of $(e_i)$ is weakly null (see~\cite[Proposition 3.2.7]{AK}); so let us fix such a sequence  $(u_n)_{n\in\N}$. By Rainwater's Theorem (\cite{Rain}, see also \cite{Ph}), it is enough to show that $\langle x^*, u_n\rangle\to0$ for every extreme point $x^*$ of the unit ball of $X_{\F,\mathbf E}^*$. The key fact is that $A:={\rm supp}(x^*)$ belongs to $\overline{\F\,}$: this will follow from the proof of Lemma \ref{ext} below.

\smallskip Assume that the family $\F$ is compact. Then $x^*$ is finitely supported, so it is clear that $\langle x^*, u_n\rangle\to0$. 

\smallskip Now, assume that $(\mathbf e_i)$ is a shrinking basis of $\mathbf E$. For $k\in\N$, set $A_k:= A\cap \llbracket 1, k\rrbracket$. Then $A_k\in  \F$ and $P_{A_k} x^*\xrightarrow{w^*} x^*$. So, for any $u\in X_{\F,\mathbf E}$, we have 
\[ \vert\langle x^*, u\rangle\vert\leqslant \sup_{F\in\F}  \vert\langle P_Fx^*, u\rangle\vert= \sup_{F\in\F}  \vert\langle x^*, P_Fu\rangle\vert.\] 
Therefore, it is enough to show that 
\[ \sup_{F\in\mathcal F}\, \vert \langle x^*, P_F(u_n)\rangle\vert \xrightarrow{n\to\infty} 0.\]
But this is clear: indeed, if $(F_n)$ is any sequence in $\F$, then 
\[ \langle x^*, P_{F_n}(u_n)\rangle=\langle \mathbf x^*, P_{F_n}(\mathbf u_n)\rangle\xrightarrow{n\to\infty} 0,\] 
because $(P_{F_n}(\mathbf u_n))$ is a bounded block sequence of $(\mathbf e_i)$ and $(\mathbf e_i)$ is shrinking.
\end{proof}

\begin{remark} Fact \ref{domi} implies in particular that $(e_i)$ is a shrinking basis of $X_{\F}$ if $\F$ is compact, and that $(e_i)$ is always a shrinking basis of $X_{\mathcal F,p}$, $1<p<\infty$.
\end{remark}

\smallskip
\begin{fact} Assume that there exist a function $\Phi:\R_+\to \R_+$ and two continuous functions $\Psi_1, \Psi_2:\R_+\to\R_+$ with $\Psi_1(t)>0$ for all $t>0$ such that 
\[ \Psi_1\left(\sum_{i\in\N} \lambda_i\right)\leqslant \left\Vert \sum_{i\in\N} \Phi(\lambda_i) \mathbf e_i\right\Vert_{\mathbf E} \leqslant \Psi_2\left(\sum_{i\in\N} \lambda_i\right)\qquad\hbox{for every $(\lambda_i)\in c_{00}^+$}.\]

Then $(e_i)$ is not a boundedly complete basis of $X_{\mathcal F, \mathbf E}$ if $\F$ is compact.
\end{fact}
\begin{proof} As mentioned in the introduction, $(e_i)$ is not a boundedly complete basis of the combinatorial Banach space $X_\F$. This means that one can find a sequence $(\lambda_i)\in \mathbb K^\N$ such that the partial sums of the series $\sum \lambda_i e_i$ are bounded in $X_\F$ but do not form a Cauchy sequence in $X_\F$. So there exist two constants $C<\infty$ and $\varepsilon >0$ and a sequence $(F_n)\subset\mathcal F$ with $F_1<F_2<\cdots $ such that, on the one hand,
\[ \sum_{i\in F} \vert \lambda_i\vert \leqslant  C\qquad\hbox{for every $F\in\mathcal F$}\]
and, on the other hand,
\[ \sum_{i\in F_n} \vert \lambda_i\vert \geqslant  \varepsilon\qquad\hbox{for all $n\in\N$}.\]

Replacing $\lambda_i$ with $\vert\lambda_i\vert$, we may assume that $\lambda_i\in\R^+$. Then, using the assumptions on $\mathbf E$, we get 
\[ \left\Vert \sum_{i\in F} \Phi(\lambda_i) \mathbf e_i\right\Vert_{\mathbf E} \leqslant  M:=\sup_{[0, C]} \Psi_2<\infty\qquad\hbox{for every $F\in\mathcal F$}\]
and
\[ \left\Vert \sum_{i\in F_n} \Phi(\lambda_i) \mathbf e_i\right\Vert_{\mathbf E} \geqslant  \alpha:=\inf_{[\varepsilon, C]} \Psi_1>0 \qquad\hbox{for all $n\in\N$}.\]

It follows that the partial sums of the series $\sum \Phi(\lambda_i) e_i$ are bounded in $X_{\mathcal F,\mathbf E}$ but do not form a Cauchy sequence; and hence that $(e_i)$ is not a boundedly complete basis of $X_{\mathcal F,\mathbf E}$.
 \end{proof}
 
 \begin{rem1} The above fact applies if $\mathbf E$ is an Orlicz sequence space $h_M$ endowed with the Luxemburg norm (see e.g. \cite{LT}), taking as $\Phi$ the inverse of the (strictly increasing) Orlicz function $M$ and setting $\Psi_1(t):=\min(1,t)$ and $\Psi_2(t):=\max(1,t)$.
 \end{rem1}
 
 \begin{rem2} It is rather plausible that $(e_i)$ is in fact never boundedly complete when $\F$ is compact, and it might even happen that $X_{\F,\mathbf E}$ is always $c_0\,$-$\,$saturated.  As mentioned in the introduction, this is true if $\mathbf E=\ell_1$; and this is also true if $\mathbf E=\ell_p$, $1<p<\infty$ and $\mathcal F$ is the Schreier family 
 $\mathcal S_1$ (\cite{BiLau}).
 \end{rem2}

\subsection{Extreme points} To prove Theorem \ref{nonsense}, we will follow  the strategy used in \cite{BFT} in the case of $X_\F^*$. The starting point is that isometries send extreme points of the unit ball to extreme points of the unit ball. So we first have to determine the extreme points of the unit ball of $X_{\mathcal F,\mathbf E}^*$; which is done in the next lemma.

\begin{lemma}\label{ext} Assume that $\mathbf E^*$ is strictly convex. Then, a functional $x^*\in X_{\mathcal F,\mathbf E}^*$ is an extreme point of $B_{X_{\mathcal F,\mathbf E}^*}$ if and only if ${\rm supp}(x^*)\in\overline{\F\,}$ and $\Vert \mathbf x^*\Vert_{\mathbf E^*}=1$.

\end{lemma}
\begin{proof} Let us define 
\[ M_1:=  \bigl\{ x^*\in X_{\F,\mathbf E}^*;\; {\rm supp}(x^*)\in\overline{\F\,}\quad{\rm and}\quad \Vert \mathbf x^*\Vert_{\mathbf E^*}=1\bigr\},\]
and 
\[ M:=  \bigl\{ x^*\in X_{\F,\mathbf E}^*;\; {\rm supp}(x^*)\in\overline{\F\,}\quad{\rm and}\quad \Vert \mathbf x^*\Vert_{\mathbf E^*}\leqslant 1\bigr\}.\]

\smallskip Note that by Corollary \ref{corsuperbasic}, we may replace $\Vert \mathbf x^*\Vert_{\mathbf E^*}$ by $\Vert x^*\Vert$ in the definition of $M$ and $M_1$. In particular, $M$ and $M_1$ are contained in $B_{X_{\mathcal F,\mathbf E}^*}$. We have to show that ${\rm Ext}(B_{X_{\mathcal F,\mathbf E}^*})=M_1$; and since extreme points of the unit ball have norm $1$, it is enough to show that $M_1\subseteq {\rm Ext}(B_{X_{\mathcal F,\mathbf E}^*})$ and ${\rm Ext}(B_{X_{\mathcal F,\mathbf E}^*})\subset M$.

\smallskip
Let $x^*\in M_1$, and assume that $x^*=\frac{u^*+v^*}2$ where $u^*,v^*\in B_{X_{\mathcal F,\mathbf E}^*}$. Then $\mathbf u^*, \mathbf v^*\in B_{\mathbf E^*}$ because $\Vert \mathbf z^*\Vert_{\mathbf E^*}\leqslant \Vert z^*\Vert$ for every $z^*\in X^*_{\mathcal F,\mathbf E}$, and $\mathbf x^*=\frac{\mathbf u^*+\mathbf v^*}2\cdot$ Since $\Vert \mathbf x^*\Vert_{\mathbf E^*}=1$ and $\mathbf E^*$ is strictly convex, it follows that $\mathbf u^*=\mathbf x^*=\mathbf v^*$ and hence $u^*=x^*=v^*$. This shows that $M_1\subseteq {\rm Ext}(B_{X_{\mathcal F,\mathbf E}^*})$.

\smallskip To show that ${\rm Ext}(B_{X_{\mathcal F,\mathbf E}^*})\subset M$, we first observe that the set $M$ is $1$-norming for $X_{\mathcal F,\mathbf E}$, \mbox{\it i.e.}
\[ \Vert x\Vert =\sup\,\bigl\{ \vert \langle z^*, x\rangle\vert;\; z^*\in M\bigr\}\qquad \hbox{for every $x\in X_{\mathcal F,\mathbf E}$}.\]
Indeed, let $x\in X_{\mathcal F,\mathbf E}$. 
If $F \in \F$ then, by $1$-unconditionality, one can find $z^*_F \in B_{X_{\mathcal F,\mathbf E}^*}$ supported on $F$ such that $\langle z^*_F, x \rangle =\langle z^*_F, P_F(x) \rangle = \Vert P_F(x) \Vert$. Note that $z^*_F \in M$ by Lemma \ref{superbasic}. Hence, 
\[ \Vert x \Vert = \sup \{\Vert P_F(x)\Vert; \; F \in \mathcal F \}= \sup \{ \langle z^*_F, P_F(x) \rangle; \; F \in \mathcal F \} \leqslant  \sup\{ \vert\langle z^*, x \rangle\vert; \; z^* \in M \}, \]
which proves our claim.
Since $M$ is obviously balanced, it follows (by the bipolar Theorem) that
\[ B_{X_{\mathcal F,\mathbf E}^*}=\overline{\rm conv}^{w^*} (M).\]

So, by Milman's converse to the Krein-Milman theorem (see e.g. \cite{Ph}), it is now enough to check that $M$ is $w^*$-$\,$closed in $B_{X_{\mathcal F,\mathbf E}^*}$. 

Let $(x_n^*)$ be a sequence in $M$ such that $x_n^*\xrightarrow{w^*} x^*\in X^*_{\mathcal F, \mathbf E}$ (even if it is not necessary, we can restrict ourselves to sequences because the bounded subsets of $X_{\F,\mathbf E}^*$ are $w^*$-$\,$metrizable by separability of $X_{\F,\mathbf E}$). Then $\Vert x^*\Vert \leqslant 1$ since $\Vert x_n^*\Vert\leqslant 1$ for all $n$ by Corollary~\ref{corsuperbasic}. So we just have to show that $x^*$ is supported on some set $A\in\overline{\F\,}$. By the definition of $M$, we know that $x_n^*$ is supported on some set $A_n\in\overline{\F\,}$ for each $n\in\N$. Upon extracting a subsequence, we may assume that $(A_n)$ converges in $\{ 0,1\}^\N$ to some set $A\in\overline{\F\,}$. Then, if $i\in\N\setminus A$, we have $i\notin A_n$ for all large enough $n$, and hence $\langle x_n^*, e_i\rangle=0$ for all large enough $n$. It follows that $\langle x^*,e_i\rangle=0$ for every $i\in\N\setminus A$, \mbox{\it i.e.} $x^*$ is supported on $A$.
\end{proof}

\begin{rem1} The proof of Lemma \ref{ext} has shown that if $\mathbf E^*$ is strictly convex, then we also have 
\[ {\rm Ext}(B_{X_{\mathcal F,\mathbf E}^*})=\bigl\{ x^*\in {X_{\mathcal F,\mathbf E}^*};\; \Vert x^*\Vert=1=\Vert \mathbf x^*\Vert_{\mathbf E^*}\bigr\}.\]

\noindent
It follows in particular that if $x^*\in X_{\mathcal F,\mathbf E}^*$ is such that $\Vert \mathbf x^*\Vert_{\mathbf E^*}=1$ and $x^*$ is \emph{not} supported on some set $A\in\overline{\F\,}$, then $\Vert x^*\Vert >1$.
\end{rem1}

\begin{rem2} If the family $\F$ is compact, the assumption that $\mathbf E^*$ is strictly convex in Lemma \ref{ext} can be weakened: it is enough to know that $\Vert \,\cdot\,\Vert_{\bf E}$ is smooth at any  finitely supported vector $\mathbf x\in S_{\mathbf E}$.
\end{rem2}
\begin{proof} With the notation of the proof of Lemma \ref{ext}, it is enough to show that $M_1\subseteq{\rm Ext}\bigl( B_{X_{\F,\mathbf E}^*}\bigr)$. Let $x^*\in M_1$ be supported on $F\in\overline{\F\,}=\F$, and choose $x\in S_{X_{\mathcal F, \mathbf E}}$ supported on $F$ such that $\langle x^*, x\rangle= 1$; this can be done since $F$ is finite and ${\rm span}\, (e_i^*,\; i\in F)$ is isometric to the dual space of ${\rm span}\,(e_i,\; i\in F)$ by $1$-unconditionality. Then $\mathbf x\in S_{\mathbf E}$ and $\mathbf x^*\in S_{\mathbf E^*}$ by Lemma \ref{superbasic}, and $\langle \mathbf x^*,\mathbf x\rangle=1$. Now, assume that 
$x^* =\frac{u^*+v^*}2$ with $u^*,v^*\in B_{X_{\mathcal F,\mathbf E}^*}$. Then $\mathbf u^*, \mathbf v^*\in B_{\mathbf E^*}$ and $\mathbf x^*=\frac{\mathbf u^*+\mathbf v^*}2\cdot$ So we have $1=\langle \mathbf x^*, \mathbf x\rangle=\frac{\langle \mathbf u^*, \mathbf x\rangle+\langle \mathbf v^*, \mathbf x\rangle}2$ and $\vert \langle \mathbf u^*, \mathbf x\rangle\vert\,,\, \vert  \langle \mathbf v^*, \mathbf x\rangle\vert\leqslant  1$, hence $ \langle \mathbf u^*, \mathbf x\rangle=1= \langle \mathbf v^*, \mathbf x\rangle$. Since  $\Vert \,\cdot\,\Vert_{\bf E}$ is smooth at $\mathbf x$, it follows that $\mathbf u^*=\mathbf x^*=\mathbf v^*$, so $u^*=x^*=v^*$.
\end{proof}

\begin{rem3} The proof of Lemma \ref{ext} has shown that without any assumption on $\bf E$, the support of every extreme point of $B_{X_{\mathcal F,\mathbf E}^*}$ belongs to $\overline{\F\,}$.
\end{rem3}

\smallskip To conclude this section, we point out another situation where we can describe the extreme points of $B_{X_{\mathcal F,\mathbf E}^*}$. The result will not be needed for the proof of Theorem \ref{nonsense}, but we use it in the proof of Theorem \ref{nonsensebis} below. 

\smallskip
Let us say that the basis $(\mathbf e_i)_{i\in\N}$ of $\mathbf E$ is \emph{strictly $1$-unconditional} if, whenever $(a_i)$, $(b_i)\in c_{00}$ are such that $\vert a_i\vert \leqslant  \vert b_i\vert$ for all $n$ with at least one strict inequality, it follows that $\Vert \sum a_i\mathbf e_i\Vert_{\bf E} <\Vert \sum b_i \mathbf e_i\Vert_{\bf E}$. For example, the canonical basis of $\ell_p$, $1\leqslant  p<\infty$ is strictly $1$-unconditional, but the canonical basis of $c_0$ is not. Also, it is not hard to see that if $\mathbf E$ is strictly convex, then $(\mathbf e_i)$ is strictly $1$-unconditional. Finally, strict $1$-unconditionality of $(\mathbf e_i)$ is equivalent to the following property: 
$\Vert P_I\mathbf x\Vert_{\bf E} < \Vert \mathbf x\Vert_{\bf E}$ for any finitely supported vector $\mathbf x\in \mathbf E$ and every $I\subsetneq {\rm supp}(\mathbf x)$.

\begin{lemma}\label{extbis} Assume that $\F$ is compact and that the basis $(\mathbf e_i)$ is strictly $1$-unconditional. For a functional $x^*\in X_{\mathcal F,\mathbf E}^*$, the following are equivalent.
\begin{enumerate}
\item[\rm (i)] $x^*$ is an extreme point of $B_{X_{\mathcal F,\mathbf E}^*}$.
\item[\rm (ii)] $x^*$ is supported on some set $F\in\F$, and 
$\mathbf x^*$ is an extreme point of the unit ball of ${\rm span}(\mathbf e_i^*, i\in G)$ for every $G\in\F$ containing $F$.
\end{enumerate}
\end{lemma}
\begin{proof} We need the following fact (that will also be used later).

\begin{fact}\label{strict} Assume that $(\mathbf e_i)$ is strictly $1$-unconditional. If $z^*\in {\rm span}(e_i^*,\; i\in\N)\subseteq X_{\mathcal F,\mathbf E}^*$ and ${\rm supp}(z^*)\notin \mathcal F$, then $\Vert z^*\Vert>\Vert \mathbf z^*\Vert_{\mathbf E^*}$. 
\end{fact}
\begin{proof}[\it Proof of Fact \ref{strict}] By assumption, the support of $z^*$ is a finite set $A\notin\mathcal F$. Since $A$ is finite and ${\rm span}\, (\mathbf e_i^*,\; i\in A)$ is isometric to the dual space of ${\rm span}\, (\mathbf e_i,\; i\in A)$, we may choose $\mathbf z\in {\rm span}\, (\mathbf e_i,\; i\in A)$ such that $\Vert \mathbf z\Vert_{\mathbf E}=1$ and 
$\langle \mathbf z^*,\mathbf z\rangle=\Vert \mathbf z^*\Vert_{\mathbf E^*}$. Denoting by $z$ the associated vector in ${\rm span}\, (e_i,\; i\in A)\subseteq X_{\F,\mathbf E}$, we have $\langle \mathbf z^*,\mathbf z\rangle=\langle  z^*, z\rangle$, and hence $\Vert \mathbf z^*\Vert_{\mathbf E^*}\leqslant  \Vert z^*\Vert \, \Vert z\Vert$. Now, by definition of the norm of $X_{\F,\mathbf E}$, we have $\Vert z\Vert=\Vert P_F(\mathbf z)\Vert_{\mathbf E}$ for some set $F\in\mathcal F$ contained in $A$; and since $A\notin \mathcal F$, the set $F$ is \emph{strictly} contained in $A$. By strict $1$-unconditionality of $(\mathbf e_i)$, it follows that $\Vert z\Vert < \Vert \mathbf z\Vert_{\mathbf E}=1$, so that $\Vert \mathbf z^*\Vert_{\mathbf E^*}< \Vert z^*\Vert$. 
\end{proof}

Assume that (ii) holds true, and let us show that $x^*$ is an extreme point of $B_{X_{\mathcal F,\mathbf E}^*}$. Note first that $\Vert x^*\Vert =\Vert \mathbf x^*\Vert_{\mathbf E^*}=1$ by Lemma \ref{superbasic}. Assume that $x^*=\frac{u^*+v^*}2$ where 
$u^*, v^*\in B_{X_{\F,\mathbf E}^*}$. Then $\mathbf x^*=\frac{\mathbf u^*+\mathbf v^*}2=\frac{ P_F\mathbf u^*+P_F\mathbf v^*}2\cdot$ So $P_F\mathbf u^*=\mathbf x^*=P_F\mathbf v^*$ by (ii). Therefore, $\Vert P_A\mathbf u^*\Vert_{\mathbf E^*}\geqslant   1$ and $\Vert P_A\mathbf v^*\Vert_{\mathbf E^*}\geqslant  1$ for any finite set $A\supset F$ by $1$-unconditionality; so any finite set $A\supset F$ contained in ${\rm supp}(\mathbf u^*)$ or in ${\rm supp}(\mathbf v^*)$ must belong to $\F$, by Fact \ref{strict}. Since the family $\F$ is compact, it follows that the supports of $\mathbf u^*$ and $\mathbf v^*$ belong to $\F$. Moreover, since $P_F\mathbf u^*=P_F\mathbf v^*$ and obviously ${\rm supp}(\mathbf u^*)\setminus F={\rm supp}(\mathbf v^*)\setminus F$, we see that $\mathbf u^*$ and $\mathbf v^*$ must have the same support. So there is a set $G\in\F$ containing $F$ such that $\mathbf u^*$ and $\mathbf v^*$ are supported on $G$. By (ii) again, it follows that $\mathbf u^*=\mathbf x^*=\mathbf v^*$, and hence $u^*=x^*=v^*$.

\smallskip Conversely, assume that $x^*\in {\rm Ext}(B_{X_{\mathcal F,\mathbf E}^*})$. Then, by the proof of Lemma \ref{ext}, $x^*$ must be supported on some set $F\in\F$. And since $\Vert \mathbf u^*\Vert_{\mathbf E^*}= \Vert u^*\Vert$ for every $u^*\in X_{\F,\mathbf E}^*$ with support in $\F$, it is clear that (ii) is satisfied.
\end{proof}

\subsection{Proof of Theorem \ref{nonsense}}  Let $J:X_{\F_1,\mathbf E}^*\to X_{\F_2,\mathbf E}^*$ be an isometry. We want to show that there exist a permutation $\pi:\N\to\N$ with $\pi(\F_1)=\F_2$ and a sequence of signs $(\omega_i)$ such that $Je_i^*=\omega_ie_{\pi(i)}^*$ for all $i\in\N$, without further assumption on $J$ if the families $\F_1$, $\F_2$ are compact, and assuming that $J$ is $w^*$-$\,w^*$ continuous or such that $J(Z_{\F_1,\mathbf E})=Z_{\F_2,\mathbf E}$ otherwise. We treat the real case and the complex case separately since the assumptions on $\mathbf E$ are different.

\medskip\noindent
{\bf Case 1.}  $\K=\R$ and $\mathbf E$ satisfies (H).

\begin{fact}\label{presupp} If $j,k\in\N$ and $\{ j,k\}\in\F_2$, then $\mathrm{supp} (J^{-1}e_j^*) \cup \mathrm{supp} (J^{-1}e_k^*)\in\overline{\F_1}$.
\end{fact}
\begin{proof}[Proof of Fact \ref{presupp}] Let $A:={\rm supp}(J^{-1}e_j^*)$ and $B:={\rm supp}(J^{-1}e_k^*)$, so that  
\[J^{-1} e_j^* = \sum_{i \in A} \lambda _i e_i^* \qquad{\rm and}\qquad J^{-1}e_k^* = \sum_{i\in B}\mu_ie_i^*\]
with $\lambda_i \neq 0$ for all $i \in A$ and $\mu_i \neq 0 $ for all $i \in B$. (If $A$ or $B$ is infinite, the corresponding series is $w^*$-$\,$convergent.) We have to show that $A\cup B\in\overline{\F_1}$.

For every $\alpha, \beta \in \K$, we have
\begin{equation}\label{truc} J^{-1}(\alpha e_j^*+ \beta e_k^*) = \sum_{i \in A \minus B} \alpha \lambda_i e_i^* + \sum_{i \in A \cap B}(\alpha \lambda_i + \beta \mu_i)e_i^* + \sum_{i \in B \minus A}\beta \mu_i e_i^*. \end{equation}
Moreover, if $\Vert \alpha e_j^*+\beta e_k^*\Vert=1$, then $\alpha e_j^*+ \beta e_k^* \in \text{Ext}(B_{X^*_{\F_2,\mathbf E}})$ by Lemma \ref{ext}, because $\{j,k\} \in \F_2$. So $ J^{-1}(\alpha e_j^*+ \beta e_k^*) \in \text{Ext}(B_{X^*_{\F_1,\mathbf E}})$, and hence $\text{supp}\bigl( J^{-1}(\alpha e_j^*+ \beta e_k^*)\bigr) \in \overline{\F_1}$. By homogeneity, it follows that $\text{supp}\bigl( J^{-1}(\alpha e_j^*+ \beta e_k^*)\bigr) \in \overline{\F_1}$ for every $(\alpha,\beta)\in\K^2$. Now, since $\K^2$ cannot be covered by countably many hyperplanes (by the Baire category Theorem), we may  choose $\alpha_0, \beta_0$ both $\neq 0$ such that $\alpha_0 \lambda_i + \beta_0 \mu_i \neq 0$ for all $i \in A \cap B$. Then $\text{supp}\bigl( J^{-1}(\alpha_0 e_j^*+ \beta_0 e_k^*)\bigr)= A \cup B$ by (\ref{truc}), and hence $A \cup B \in \overline{\F_1}$.
\end{proof}

\begin{remark} Fact \ref{presupp} is the only place in the proof where we use the strict convexity of $\mathbf E^*$ (which is needed in Lemma \ref{ext}). If the families $\F_1$ and $\F_2$ are compact, it is enough to assume that $\Vert\,\cdot\,\Vert_{\bf E}$ is smooth at every finitely supported vector $\mathbf x\in S_{\bf E}$, by the second remark following the proof of Lemma \ref{ext}.

\end{remark}

\begin{fact}\label{supp} If $j\neq k$ and $\{j,k\} \in \F_2$, then $\mathrm{supp} (J^{-1}e_j^*) \cap \mathrm{supp} (J^{-1}e_k^*) = \emptyset$.
  \end{fact}
\begin{proof}[Proof of Fact \ref{supp}] Let again $A:={\rm supp}(J^{-1}e_j^*)$ and $B:={\rm supp}(J^{-1}e_k^*)$, so 
\[J^{-1} e_j^* = \sum_{i \in A} \lambda _i e_i^* \qquad{\rm and}\qquad J^{-1}e_k^* = \sum_{i\in B}\mu_ie_i^*\]
with $\lambda_i \neq 0$ for all $i \in A$ and $\mu_i \neq 0 $ for all $i \in B$. We have to show that $A\cap B=\emptyset$.

By Fact \ref{presupp}, we know that $A \cup B \in \overline{\F_1}$. 
Now, recall that if $\F$ is a combinatorial family, $I\in\overline{\F\,}$ and $x^*=\sum_{i\in I} x_i e_i^*$ is a vector of $X^*_{\F,\mathbf E}$ supported on $I$ (where the series is $w^*$-$\,$convergent), then the series $\sum_{i\in I} x_i\mathbf e_i^*$ is $w^*$-$\,$convergent in $\mathbf E^*$ and, by Corollary~\ref{corsuperbasic},  
$\Vert x^*\Vert=\Vert \sum_{i\in I} x_i \mathbf e^*_i\Vert_{\mathbf E^*}$. Hence, if we define the operator $T: {\rm span}(\mathbf e_j^*, \mathbf e_k^*)\to \overline{{\rm span}(\mathbf e_i^*, i\in A\cup B)}^{\,w^*}\subseteq\mathbf E^*$ by 
\[ T\mathbf e_j^*:=\sum_{i \in A} \lambda _i \mathbf e_i^*\qquad{\rm and} \qquad T\mathbf e_k^*:=\sum_{i \in B} \mu_i \mathbf e_i^*,\]
then $T$ is an isometric embedding for the norm $\Vert\,\cdot\,\Vert_{\mathbf E^*}$. By assumption (H), it follows that ${\rm supp}(T\mathbf e_j^*)\cap{\rm supp}(T \mathbf e_k^*)=\emptyset$, \textit{i.e.} $A\cap B=\emptyset$ as required.
\end{proof}

\begin{fact}\label{singleton} For any $n \in \N$, the support of $Je_n^*$ is a singleton. 
\end{fact}
\begin{proof} 
The argument is exactly the same as in \cite{BFT}.  Let $A:={\rm supp}(Je_n^*)$, and write
\[ Je_n^* = \sum_{i \in A} \lambda_i e_i^*,\]
where  $\lambda_i \neq 0$ for all $i \in A$ and the series is $w^*$-$\,$convergent if $A$ is infinite. 

\smallskip
By  Fact \ref{presupp} applied to $J^{-1}$ and $j=n=k$, we know that $A \in \overline{\F_2}$. Hence, Fact~\ref{supp} and the hereditarity property imply that for any $j,k \in A $ such that $j \neq k$, we have $\text{supp} (J^{-1}e_j^*) \cap \text{supp} (J^{-1}e_k^*) = \emptyset$. 

\smallskip Assume that $J$ is $w^*$-$\, w^*$ continuous. Then so is $J^{-1}$. Hence, we may write 
\[ e_n^*=J^{-1}(Je_n^*)=\sum_{i\in A} \lambda_i J^{-1}e_i^*;\]
and since $\lambda_i\neq 0$ for all $i\in A$ and the $J^{-1}e_i^*$ have disjoint supports, it follows that $A$ must be a singleton.

\smallskip Now assume that $J(Z_{\F_1,\mathbf E})=Z_{\F_2,\mathbf E}$. Then the series $\sum_{i\in A} \lambda_i e_i^*$ is convergent in norm (if $A$ is infinite) since $Je_n^*\in Z_{\F_2,\mathbf E}$ and  $(e_i^*)$ is a Schauder basis of $Z_{\F_2,\mathbf E}$; so one gets as above that $A$ is a singleton.

\smallskip Finally, assume that $\F_1$ and $\F_2$ are compact. Then $(e_i)$ is a shrinking basis of $X_{\F_1,\mathbf E}$ and $X_{\F_2,\mathbf E}$ by Fact \ref{domi}, so $J(Z_{\F_1,\mathbf E})=Z_{\F_2,\mathbf E}$. Moreover, with the notation of the proof of Fact \ref{supp}, the set $A\cup B$ is finite; so it is enough to assume that (H) holds true for isometric embeddings $T: {\rm span}(\mathbf e_j^*, \mathbf e_k^*)\to {\rm span}(\mathbf e_i^*, i\in \N)$.
\end{proof}

It is now easy to conclude the proof. By Fact \ref{singleton} and since $\Vert Je^*_n\Vert =\Vert e^*_n\Vert =1$ for all $n\in\N$, there is a sequence of signs $(\omega_n)$ and a map $\pi:\N\to\N$ such that $Je^*_n=\omega_n e^*_{\pi(n)}$ for all $n\in\N$. Since the same holds true for $J^{-1}$, the map $\pi$ has to be a bijection. Finally, since $J$ and $J^{-1}$ preserve extreme points, we have $\pi(\F_1)=\F_2$ by Lemma \ref{ext}.

\medskip\noindent
{\bf Case 2.} $\K=\mathbb C$ and ${\rm span}(\mathbf e_j^*,\mathbf e_k^*)$ is not isometrically equal to $\ell_2(2)$ for any $j\neq k$.

\medskip In this case, we use as a blackbox the theory developed in \cite{KW}. More precisely, what we need from \cite{KW} can be summarized as follows.

\smallskip
- If $Z$ is a complex Banach space endowed  with a normalized $1$-unconditional basis $(f_i)_{i\in\N}$, then there is a partition $(I_\lambda)_{\lambda\in\Lambda}$ of $\N$ such that the spaces $H_\lambda:=\overline{\rm span}\,(f_i;\; i\in I_\lambda)$ are the \emph{maximal hermitian subspaces} of $Z$ (one does not need to know the meaning of ``hermitian subspace''). These spaces $H_\lambda$ are isometrically Hilbertian and  are called the \emph{Hilbert components} of $Z$. Moreover, if $f_j$ and $f_k$ are in the same Hilbert component, then ${\rm span}(f_j,f_k)$ is isometrically equal to 
$\ell_2(2)$.

\smallskip
- If $Z$ and $Z'$ are two such spaces with Hilbert components $(H_\lambda)_{\lambda\in\Lambda}$ and $(H_{\lambda'})_{\lambda'\in\Lambda'}$ and if $J:Z\to Z'$ is an isometry, then there is a bijection $\pi:\Lambda\to\Lambda'$ such that $J H_\lambda=H_{\pi(\lambda)}$ for every $\lambda\in\Lambda$.

\medskip
In our case, the Hibert components of $Z:= Z_{\F_1,\mathbf E}$ and $Z':= Z_{\F_2,\mathbf E}$ are $1$-dimensional. Indeed, assume for example that some Hilbert component of $Z_{\F_1,\mathbf E}$ is not $1$-dimensional. Then, one can find $j\neq k$ such that ${\rm span}(e_j^*,e_k^*)$ is isometrically equal to $\ell_2(2)$. 
 However, 
if $\{ j,k\}\in\mathcal F_1$ then ${\rm span}(e_j^*,e_k^*)$ is canonically isometric to ${\rm span}(\mathbf e_j^*, \mathbf e_k^*)$ by Lemma \ref{superbasic}, so ${\rm span}(e_j^*,e_k^*)$ is not isometrically equal to $\ell_2(2)$ 
by assumption; and if $\{ j,k\}\notin\mathcal F_1$, then ${\rm span}(e_j^*,e_k^*)$ is not even Hilbertian since it is isometric to the dual space of ${\rm span}(e_j,e_k)$ and the latter is isometrically equal to $\ell_\infty(2)$. 

So the Hilbert components of $Z_{\F_1,\mathbf E}$ and $Z_{\F_2,\mathbf E}$ are just the linear spans of the vectors $e_n^*$. By \cite{KW}, it follows that if the isometry $J: X_{\F_1,\mathbf E}^*\to X_{\F_2,\mathbf E}^*$ is such that $J(Z_{\F_1,\mathbf E})=Z_{\F_2,\mathbf E}$, then there is a permutation $\pi:\N\to \N$ and signs $\omega_n$ such that $Je_n^*=\omega_n e_{\pi(n)}^*$ for all $n\in\N$; and the permutation $\pi$ must be such that $\pi(\mathcal F_1)=\mathcal F_2$ by Lemma \ref{ext}.

\smallskip Since ${\rm span}(\mathbf e_j,\mathbf e_k)$ is not isometrically equal to $\ell_2(2)$ for any $j\neq k$, one shows in the same way that any isometry between  $X_{\F_1,\mathbf E}$ and $X_{\F_2,\mathbf E}$ is a signed permutation of $(e_i)$; from which it follows that if the isometry $J: X_{\F_1,\mathbf E}^*\to X_{\F_2,\mathbf E}^*$ is $w^*$-$\,w^*$ continuous then it has the required form.

\smallskip Finally, we observed above that if $\F_1$ and $\F_2$ are compact then $J(Z_{\F_1,\mathbf E})=Z_{\F_2,\mathbf E}$. 

\subsection{Proof of Proposition \ref{l1}} In this section, we use \cite{KW} to prove a general result that immediately implies Proposition \ref{l1}. Recall that the basis $(\mathbf e_i)_{i\in\N}$ of $\mathbf E$ is said to be {strictly $1$-unconditional} if, whenever $(a_i)$, $(b_i)\in c_{00}$ are such that $\vert a_i\vert \leqslant  \vert b_i\vert$ for all $n$ with at least one strict inequality, it follows that $\Vert \sum a_i\mathbf e_i\Vert_{\bf E} <\Vert \sum b_i \mathbf e_i\Vert_{\bf E}$. Recall also the notation $Z_{\F,\mathbf E}= \overline{{\rm span}(e_i^*,\; i\in\N)}^{X_{\mathcal F,\mathbf E}^*}$. Finally, if $\F_1$ and $\F_2$ are two combinatorial families, we say that a permutation $\pi:\N\to\N$ is \emph{$(\F_1,\F_2)\,$-$\,$compatible} if $\pi(\F_1)=\F_2$.

\begin{proposition}\label{trucmachin} Let $\mathcal F_1$ and $\mathcal F_2$ be two combinatorial families.
\begin{enumerate}
\item[\rm (1)] If the basis $(\mathbf e_i)$ is strictly $1$-unconditional, then any permutation $\pi:\N\to \N$ inducing an isometry $J:X_{\mathcal F_1,\mathbf E}\to X_{\mathcal F_2, \mathbf E}$ must be $(\F_1,\F_2)\,$-$\,$compatible; and if $(\mathbf e_i)$ is both strictly $1$-$\,$unconditional and $1$-$\,$symmetric, then any permutation inducing an isometry $J:Z_{\F_1,\mathbf E}\to Z_{\F_2,\mathbf E}$ is $(\F_1,\F_2)\,$-$\,$compatible.
\item[\rm (2)] Assume that  $\K=\mathbb C$. If ${\rm span}(\mathbf e_j,\mathbf e_k)$ is not isometrically equal to $\ell_2(2)$ for any $j\neq k$, then any isometry $J:X_{\mathcal F_1,\mathbf E}\to X_{\mathcal F_2, \mathbf E}$ sign-permutes $(e_i)$, and any isometry 
$J:Z_{\F_1,\mathbf E}\to Z_{\F_2,\mathbf E}$ sign-permutes $(e_i^*)$. 
\end{enumerate}
\end{proposition} 
\begin{proof}  (1) Assume that $(\mathbf e_i)$ is strictly $1$-unconditional, and let  $\pi:\N\to \N$  be a permutation inducing an isometry $J:X_{\mathcal F_1,\mathbf E}\to X_{\mathcal F_2, \mathbf E}$. It is enough to show that $\pi(\mathcal F_1)\subseteq \mathcal F_2$ (since we can then apply the result to $\pi^{-1}$); so we fix $F\in\mathcal F_1$ and we show that $\pi(F)\in\mathcal F_2$. 

By assumption, for any $a=(a_i)\in c_{00}$, we have 
\[ \left\Vert \sum_{i\in F} a_{\pi(i)} \mathbf e_i\right\Vert_{\mathbf E} =  \left\Vert \sum_{i\in F} a_{\pi(i)}  e_i\right\Vert_{X_{\mathcal F_1,\mathbf E}}=\left\Vert \sum_{i\in F} a_{\pi(i)} e_{\pi(i)}\right\Vert_{X_{\mathcal F_2,\mathbf E}}= \left\Vert \sum_{j\in \pi(F)} a_{j}  e_j\right\Vert_{X_{\mathcal F_2,\mathbf E}}.\]

In particular, taking $a:=\mathbf 1_{\pi(F)}$ we get
\[ \left\Vert \sum_{i\in F} \mathbf e_i\right\Vert_{\mathbf E}= \left\Vert \sum_{j\in \pi(F)}  e_j\right\Vert_{X_{\mathcal F_2, \mathbf E}}.\]

By definition of the norm $\Vert\,\cdot\,\Vert_{X_{\mathcal F_2, \mathbf E}}$, it follows that one can find a set $F'\in \mathcal F_2$ such that 
\[ F'\subseteq \pi(F)\qquad{\rm and}\qquad \left\Vert \sum_{i\in F}  \mathbf e_i\right\Vert_{\mathbf E} = \left\Vert \sum_{j\in F'} \mathbf e_j\right\Vert_{\mathbf E}.\]

One can do the same with $\pi^{-1}$ and the set 
$F'\in\mathcal F_2$: this gives $F''\in\mathcal F_1$ such that 
\[  F''\subseteq \pi^{-1}(F')\qquad{\rm and}\qquad \left\Vert \sum_{j\in F'} \mathbf e_j\right\Vert_{\mathbf E}= \left\Vert \sum_{i\in F''} \mathbf e_i\right\Vert_{\mathbf E}.\]

So we obtain 
\[ \left\Vert \sum_{i\in F}  \mathbf e_i\right\Vert_{\mathbf E}= \left\Vert \sum_{i\in F''} \mathbf e_i\right\Vert_{\mathbf E}.\]

By strict $1$-unconditionality of the basis $(\mathbf e_i)$, it follows that $\pi(F)\subseteq F'$, since otherwise $F''$ is strictly contained in $F$. Hence $\pi(F)\in\mathcal F_2$.

\smallskip Now assume that $(\mathbf e_i)$ is both strictly $1$-unconditional and $1$-symmetric, and let $\pi$ be a permutation inducing an isometry $J: Z_{\F_1,\mathbf E}\to Z_{\F_2,\mathbf E}$.

If $F\in\mathcal F_1$ then, as above, we see that
\[ \left\Vert \sum_{i\in F} \mathbf e_i^*\right\Vert_{\mathbf E^*}= \left\Vert \sum_{j\in \pi(F)}  e_j^*\right\Vert_{X_{\mathcal F_2, \mathbf E}^*}.\]

By Fact \ref{strict}, it follows that if $\pi(F)\notin\mathcal F_2$, then $\left\Vert \sum_{i\in F} \mathbf e_i^*\right\Vert_{\mathbf E^*}> \left\Vert \sum_{j\in \pi(F)} \mathbf e_j^*\right\Vert_{\mathbf E^*}$; which is a contradiction since $(\mathbf e_i^*)$ is a $1$-symmetric basis of its closed linear span.

\smallskip (2) As in the proof of Theorem \ref{nonsense}, one checks that the Hilbert components of $X_{\mathcal F_1,\mathbf E}$ and $X_{\mathcal F_2, \mathbf E}$ are $1$-dimensional. By \cite{KW}, it follows that any isometry $J:X_{\mathcal F_1,\mathbf E}\to X_{\mathcal F_2, \mathbf E}$ 
must sign-permute $(e_i)$. And since ${\rm span}(\mathbf e_j^*,\mathbf e_k^*)$ is not isometrically equal to $\ell_2(2)$ for any $j\neq k$, the analogous conclusion holds for isometries $J:Z_{\F_1,\mathbf E}\to Z_{\F_2,\mathbf E}$. 
\end{proof}

\begin{remark} If we take $\mathbf E:=c_0$, then $X_{\mathcal F_1,\mathbf E}=c_0= X_{\mathcal F_2, \mathbf E}$ regardless of the families $\mathcal F_1$, $\mathcal F_2$, and hence any permutation $\pi:\N\to \N$ induces an isometry $J:X_{\mathcal F_1,\mathbf E}\to X_{\mathcal F_2, \mathbf E}$. This shows that \emph{some} assumption on $\mathbf E$ is needed to ensure that a permutation $\pi$ inducing an isometry between $X_{\mathcal F_1,\mathbf E}$ and $X_{\mathcal F_2, \mathbf E}$ must satisfy $\pi(\mathcal F_1)=\mathcal F_2$. 
\end{remark}

\smallskip From Proposition \ref{trucmachin}, we immediately obtain the following general version of Proposition \ref{l1}. 
\begin{corollary}\label{l1bis} Assume that $\K=\mathbb C$. Moreover, assume that the basis $(\mathbf e_i)$ of $\mathbf E$  is strictly $1$-$\,$unconditional, and that ${\rm span}(\mathbf e_j,\mathbf e_k)$ is not isometrically equal to $\ell_2(2)$ for any $j\neq k$. If $\F_1$ and $\F_2$ are two combinatorial families, then any isometry $J:X_{\F_1,\mathbf E}\to X_{\F_2,\mathbf E}$ is an $(\F_1,\F_2)\,$-$\,$compatible signed permutation of $(e_i)$. If $(\mathbf e_i)$ is $1$-$\,$symmetric and either $\F_1$ and $\F_2$ are compact or $(\mathbf e_i)$ is shrinking, then any isometry $J:X_{\F_1,\mathbf E}^*\to X_{\F_2,\mathbf E}^*$ is an $(\F_1,\F_2)\,$-$\,$compatible signed permutation of $(e_i^*)$.
\end{corollary}
\begin{proof} The first part is clear by Proposition \ref{trucmachin}. For the second part, note that in both cases $(e_i)$ is a shrinking basis of $X_{\F_1,\mathbf E}$ and $X_{\F_2,\mathbf E}$ by Fact \ref{domi}, \textit{i.e.} $Z_{\F_1,\mathbf E}=X_{\F_1,\mathbf E}^*$ and $Z_{\F_2,\mathbf E}=X_{\F_2,\mathbf E}^*$.
\end{proof}

\begin{remark} The result is not true in the real case: C.  Brech and A. Tcaciuc \cite{AC} have found a simple example of a compact combinatorial family $\F$ for which not all isometries of $X_{\F}=X_{\F,\ell_1}$ are signed permutations of $(e_i)$.
\end{remark}

\subsection{A variant of Theorem \ref{nonsense}} In this section, we prove the following variant of Theorem \ref{nonsense}. 
\begin{theorem}\label{nonsensebis} Assume that the basis $(\mathbf e_i)$ is strictly $1$-unconditional, and that either $\K=\R$ and $\mathbf E$ satisfies {\rm (H)}, or $\K=\mathbb C$ and ${\rm span}(\mathbf e_j^*, \mathbf e_k^*)$ is not isometrically equal to $\ell_2(2)$ for any $j\neq k$. Moreover, assume that for any finite set $F\subset\N$, the extreme points of the unit ball of ${\rm span}(\mathbf e_i^*,\; i\in F)$ cannot be covered by finitely many hyperplanes of 
${\rm span}(\mathbf e_i^*,\; i\in F)$. Then, if $\F_1$ and $\F_2$ are compact combinatorial families, any isometry $J:X_{\F_1,\mathbf E}^*\to X_{\F_2,\mathbf E}^*$ is an $(\F_1,\F_2)\,$-$\,$compatible signed permutation of $(e_i^*)$.
\end{theorem}
\begin{proof} Let us fix an isometry $J:X_{\F_1,\mathbf E}^*\to X_{\F_2,\mathbf E}^*$. Note that, since $(e_i)$ is a shrinking basis of $X_{\F_1,\mathbf E}$ and $X_{\F_2,\mathbf E}$ by Fact \ref{domi}, we have $J(Z_{\F_1,\mathbf E})= Z_{\F_2,\mathbf E}$.

\smallskip 
Assume that we have been able to show that $J$ sign-permutes the basis $(e_i^*)$, and let us show the associated permutation $\pi:\N\to\N$ must be $(\F_1, \F_2)\,$-$\,$compatible. It is enough to check that $\pi(F)\in\F_2$ for any set $F\in\F_1^{MAX}$ (this is enough to conclude that 
$\pi(\F_1)\subseteq \F_2$, and then one can apply the result to $\pi^{-1}$). So let  
$F\in\F_1^{MAX}$. By assumption, one can find an extreme point of the unit ball of ${\rm span}(\mathbf e_i^*,\; i\in F)$ with support exactly equal to $F$, say $\mathbf x^*=\sum_{i\in F} \alpha_i \mathbf e_i^*$ . Then $x^*:=\sum_{i\in F} \alpha_i e_i^*$ is an extreme point of $B_{X_{\F_1,\mathbf E}^*}$ by Lemma \ref{extbis}. So $Jx^*=\sum_{i\in F} \alpha_i \omega_i e_{\pi(i)}^*$ is an extreme point of  $B_{X_{\F_2,\mathbf E}^*}$; and by Lemma \ref{extbis} again, it follows that $\pi(F)\in \F_2$.

\smallskip Now, let us show that $J$ sign-permutes $(e_i^*)$.

\smallskip
(i) We first note that $J^{-1}e_n^*$ is finitely supported for every $n\in\N$. Indeed, since $\F_2$ is compact, we can choose $F\in\F_2^{MAX}$ such that $n\in F$. Then $e_n^*$ is a linear combination of extreme points of the unit ball of the finite-dimensional space ${\rm span}(e_i^*,\, i\in F)$. By Lemma \ref{extbis}, these points are extreme points of $B_{X_{\F_2,\mathbf E}^*}$. So $J^{-1}e_n^*$ is a linear combination of extreme points of $B_{X_{\F_1,\mathbf E}^*}$, and hence $J^{-1}e_n^*$ is finitely supported by Lemma \ref{extbis} again.

\smallskip (ii) Assume that $\K=\R$ and $\mathbf E$ satisfies {\rm (H)}. Looking back at the proof of Theorem~\ref{nonsense}, we see that the only thing to check is that if $j,k\in\N$ and $\{ j,k\}\in\F_2$, then ${\rm supp}(J^{-1}e_j^*)\cup {\rm supp}(J^{-1}e_k^*)\in\F_1$. 

Let us choose $F\in\F_2^{MAX}$ such that $\{ j,k\}\subseteq F$, and let $G:=\bigcup_{n\in F} {\rm supp}(J^{-1}e_n^*)$. Note that $G$ is a finite set by (i). Since $\F_1$ is hereditary, it is enough to show that $G\in\F_1$. 

For any $n\in F$, we may write 
\[ J^{-1}e_n^*=\sum_{i\in G} \lambda_{i,n} e_i^*\]
and, for every $i\in G$, at least one $\lambda_{i,n}$ is $\neq 0$. 
Then, for any $x^*=\sum_{n\in F} \alpha_n e_n^*$, we have 
\[ J^{-1}x^*=\sum_{i\in G} \left(\sum_{n\in F} \lambda_{i,n} \alpha_n\right) e_i^*.\]
By assumption, we can choose $x^*=\sum_{n\in F} \alpha_n e_n^*$ in such a way that $\mathbf x^*$ is an extreme point  of the unit ball of ${\rm span}(\mathbf e_i^*,\; i\in F)$ and $\sum_{n\in F}\lambda_{i,n}\alpha_n\neq 0$ for all $i\in G$, \textit{i.e.} ${\rm supp}(J^{-1}x^*)=G$. By Lemma \ref{extbis}, $x^*\in {\rm Ext}\bigl( B_{X_{\F_2,\mathbf E}^*}\bigr)$. So $J^{-1}x^*\in {\rm Ext}\bigl( B_{X_{\F_1,\mathbf E}^*}\bigr)$, and hence $G\in \F_1$ by Lemma \ref{extbis} again.

\smallskip (iii) Assume that $\K=\mathbb C$ and that ${\rm span}(\mathbf e_j^*, \mathbf e_k^*)$ is not isometrically equal to $\ell_2(2)$ for any $j\neq k$. Then, since $J(Z_{\F_1,\mathbf E})= Z_{\F_1,\mathbf E}$, one shows exactly as in the proof of Theorem~\ref{nonsense} that $J$ sign-permutes $(e_i^*)$.
\end{proof}

\begin{remark} Theorem \ref{nonsensebis} can be applied to $\mathbf E:=\ell_1$ in the complex case, but not in the real case. This is as it should be since, as already mentioned, the conclusion is simply false in the real case (\cite{AC}).
\end{remark}

\subsection{Three examples}\label{Exs} In this section, we give three examples in order to illustrate the assumptions in Theorem \ref{nonsense} and Corollary \ref{l1bis}. In these examples, the space $\mathbf E$ is the Orlicz sequence space $h_M$ associated with some Orlicz function $M:[0,\infty)\to [0,\infty)$.

\smallskip We follow the notation and definitions of \cite{LT} regarding Orlicz functions $M$ and Orlicz sequence spaces $\ell_M$ and $h_M$, and we use them without comment. However, let us recall the definition of the two natural norms on $\ell_M$, the \emph{Luxemburg norm} $\Vert\,\cdot\,\Vert_M$ and the \emph{Orlicz norm} $\Vert \,\cdot\,\Vert_M^*$: if $x=(x_i)_{i\in\N}\in \ell_M$, then 
\[ \Vert x\Vert_M=\inf\,\left\{ \lambda>0;\; \sum_{i=1}^\infty M\left(\frac{\vert x_i\vert}\lambda\right)\leqslant  1\right\}\]
and
\[ \Vert x\Vert^*_M=\sup\left\{  \sum_{i=1}^\infty\vert x_i y_i\vert;\; \sum_{i=1}^\infty M^*(\vert y_i\vert)\leqslant 1\right\},\]
where $M^*$ is the function complementary to $M$.

\smallskip
We require that all Orlicz functions $M$ are ``normalized'', \textit{i.e.} satisfy $M(1)=1$, so that the canonical basis $(\mathbf e_i)_{i\in\N}$ of $h_M$ is normalized (and, of course, $1$-symmetric) when $h_M$ is endowed with either the Luxemburg norm or the Orlicz norm. We also require that all Orlicz functions are ``non-trivial'' in the sense that $M(t)>0$ for $t>0$,  $\lim_{t\to 0} M'(t)=0$ and $\lim_{t\to\infty} M'(t)=\infty$, where $M'$ is the right derivative of $M$. Then $M^*$ is also a non-trivial Orlicz function such that $M^*(1)=1$.

\begin{example} Let $M$ be an Orlicz function such that $M$ and its complementary Orlicz function $M^*$ satisfy the $\Delta_2$ condition at $0$, and let $\mathbf E$ be the Orlicz sequence space $\ell_M=h_M$ endowed with the Orlicz norm. Assume that $M$ is $\mathcal C^2$-$\,$smooth on $(0,\infty)$ with $M''(u)>0$, and that $\liminf_{u\to 0} \frac{uM''(u)}{M'(u)}>0$. Let also $\F_1$ and $\F_2$ be two combinatorial families. If either $M''(u)\to 0$ or $M''(u)\to \infty$ as $u\to 0$, then any isometry $J:X_{\F_1,\mathbf E}^*\to X_{\F_2,\mathbf E}^*$ is an $(\F_1,\F_2)\,$-$\,$compatible signed permutation of $(e_i^*)$.
\end{example}
\begin{proof} Since $M$ and $M^*$ satisfy the $\Delta_2$ condition at $0$, we know that $\mathbf E$ is reflexive and that $\mathbf E^*$ is the space $\ell_{M^*}=h_{M^*}$ endowed with the Luxemburg norm. 

Since $M$ is $\mathcal C^1$-$\,$smooth and $M'(u)>0$ on $(0,\infty)$, the space $\mathbf E=h_M$ is smooth (see \cite[Theorem 7.2.5]{RR}), and hence $\mathbf E^*$ is strictly convex by reflexivity. Moreover, since $\bf E$ is reflexive, $(e_i)$ is a shrinking basis of $X_{\F,\mathbf E}$ for any combinatorial family $\F$, by Fact \ref{domi}. So, by Theorem \ref{nonsense}, it is enough to show that if $j\neq k\in\N$ and $T:{\rm span}(\mathbf e_j^*, \mathbf e_k^*)\to \mathbf E^*$ is a linear isometric embedding, then $T\mathbf e_j^*$ and $T\mathbf e_k^*$ have disjoint supports.

Note that since $M''(u)>0$ on $(0,\infty)$, the Orlicz complementary function $M^*$ is $\mathcal C^1$-$\,$smooth on $[0,\infty)$ with $(M^*)'(t)=(M')^{-1}(t)$, and $\mathcal C^2$-$\,$smooth on $(0,\infty)$ with \[ (M^*)''(t)=\frac1{M''((M')^{-1}(t))}\cdot\]
It follows that 
\[ \limsup_{t\to 0} \frac{t(M^*)''(t)}{(M^*)'(t)}=\limsup_{u\to 0} \frac{M'(u)}{uM''(u)}<\infty.\]

In the terminology of \cite{Beata}, this means that $M^*$ satisfies the condition $\Delta_{2+}$ at $0$. Hence, the result follows from \cite{Beata}; more precisely, from the proofs of the main results of \cite{Beata}.
\end{proof}

\smallskip For the other two examples, we need the following fact, which is certainly well known.
\begin{fact}\label{2dim} Let $M$ be an Orlicz function, and let $\mathbf E=h_M$ be endowed with either the Luxemburg norm or the Orlicz norm. Then ${\rm span}(\mathbf e_1, \mathbf e_2)$ is isometrically equal to $\ell_2(2)$ if and only if $M(t)+M(\sqrt{1-t^2})=1$ for all $t\in (0,1)$.
\end{fact}
\begin{proof} By duality, it is enough to prove this in the case of the Luxemburg norm. By definition of the Luxemburg norm (and since $M$ is continuous and strictly increasing), if $a,b\in\K$ and $(a,b)\neq (0,0)$ then 
\[ \Vert a\mathbf e_1+b\mathbf e_2\Vert_M=M_{a,b}^{-1}(1),\]
where $M_{a,b}:(0,\infty)\to (0,\infty)$ is defined by $M_{a,b}(\lambda):= M(\vert a\vert/\lambda)+M(\vert b\vert/\lambda)$. So,  ${\rm span}(\mathbf e_1, \mathbf e_2)$ is isometrically equal to $\ell_2(2)$ if and only if 
\[ M\left(\frac{a}{\sqrt{a^2+b^2}}\right)+M\left(\frac{b}{\sqrt{a^2+b^2}}\right)=1\qquad\hbox{for every $a,b>0$}.\]
\end{proof}

\begin{example} Assume that $\K=\mathbb C$. Let $M$ be an Orlicz function such that $M$ and its complementary Orlicz function $M^*$ satisfy the $\Delta_2$ condition at $0$, and let $\mathbf E$ be the Orlicz sequence space $\ell_M=h_M$ endowed with either the Luxemburg norm or the Orlicz norm. Let also $\F_1$ and $\F_2$ be two combinatorial families. If there exists $t\in(0,1)$ such that $M(\sqrt{1-t^2})\neq 1-M(t)$, then any isometry $J:X_{\F_1,\mathbf E}^*\to X_{\F_2,\mathbf E}^*$ is an $(\F_1,\F_2)\,$-$\,$compatible signed permutation of $(e_i^*)$. This holds in particular if $\lim_{t\to 0^+} {M'(t)}/t\neq \lim_{u\to 1^-} M'(u)$.
\end{example}
\begin{proof} Since $M$ and $M^*$ are continuous and strictly increasing, it is clear that $(\mathbf e_i)$ is strictly $1$-unconditional. Moreover $(\mathbf e_i)$ is also shrinking because $\bf E$ is reflexive. Finally, ${\rm span}(\mathbf e_1, \mathbf e_2)$ is not isometrically equal to $\ell_2(2)$ by Fact \ref{2dim}. So we may apply Corollary~\ref{l1bis}. For the final assertion, note that if $M(\sqrt{1-t^2})= 1-M(t)$ for all $t\in (0,1)$ then $M'(t)/t=M'(\sqrt{1-t^2})/\sqrt{1-t^2}$, and hence $\lim_{t\to 0^+} {M'(t)}/t= \lim_{u\to 1^-} M'(u)$.
\end{proof}

\smallskip
\begin{example} Let $c:=1/(2-\sqrt{2})$, and consider the function $M:[0,\infty)\to [0,\infty)$ defined as follows:
\[ M(t):=\left\{ \begin{matrix} c\, ({1-\sqrt{1-t^2}})& &0\leqslant t\leqslant 1/\sqrt{2},\\
\displaystyle ct+1-c & &1/\sqrt2<t\leqslant 1,\\
\displaystyle c\,\frac{t^2}2+1-\frac{c}2& & t>1.
\end{matrix}
\right.
\]
This is an Orlicz function satisfying the $\Delta_2$ condition at $0$. Let $\mathbf E$ be the Orlicz sequence space $\ell_M=h_M$ endowed with the Luxemburg norm. Then $\mathbf E^*$ is strictly convex and $\mathbf E$ is not a Hilbert space. However, ${\rm span}(\mathbf e_1,\mathbf e_2)$ is isometrically equal to 
$\ell_2(2)$. Consequently, for some compact combinatorial families $\F$, there are isometries of $X_{\F,\mathbf E}$ which are not signed permutations of $(e_i)$.
\end{example}
\begin{proof} This example is taken from \cite{Gr}. Note that the definition of $M(t)$ for $t>1$ is what it is only because we want that $M'(t)\to\infty$ as $t\to\infty$.

\smallskip  The space $\mathbf E$ is smooth because $M$ is $\mathcal C^1$-$\,$smooth with $M'(t)>0$ on $(0,\infty)$, see \cite[Theorem 7.2.3]{RR}. Moreover, $\mathbf E$ is isomorphic to $\ell_2$ because $M(t)$ behaves like $t^2$ when $t\to 0$;  in particular $\mathbf E$ is reflexive. Hence, 
$\mathbf E^*$ is strictly convex.

\smallskip Since $M$ is not the function $t^2$ on $[0,1]$, it is ``clear'' that $\mathbf E=h_M$ is not a Hilbert space. However, let us give a detailed proof. So, assume that $M$ is an Orlicz function such that $\mathbf E=h_M$ is a Hilbert space, and let us show that $M(t)=t^2$ on $[0,1]$. For simplicity, we give the proof assuming additionally that $M$ is differentiable. Since $(\mathbf e_i)$ is a normalized $1$-unconditional basis of the Hilbert space $\mathbf E$, it is in fact an orthonormal basis of $\mathbf E$. In particular, for any $a,b,c\in\K$, we have 
\[ \Vert a \mathbf e_1+b\mathbf e_2+c\mathbf e_3\Vert_{\mathbf E}^2= \vert a\vert^2+\vert b\vert^2+\vert c\vert ^2.\]
By definition of the Luxemburg norm, this means that for any $(x,y,z)\in\R^3$, the following implication holds:
\[ x^2+y^2+z^2=1\implies M(x)+M(y)+M(z)=1.\]
In other words, for any $x,y\in\R^2$ such that $x^2+y^2\leqslant 1$, we have 
\[ M(x)+M(y)+M(\sqrt{1-x^2-y^2})=1.\]
Differentiating with respect to $x$ and setting $u:=\sqrt{1-x^2-y^2}$ it follows that $\frac{M'(x)}{x}=\frac{M'(u)}{u}$ whenever $0< x<1$ and $u\leqslant \sqrt{1-x^2}$. This implies that the function $\frac{M'(t)}{t}$ is constant on $(0,1)$; and since $M(0)=0$ and $M(1)=1$, we conclude that $M(t)=t^2$ on $[0,1]$.

\smallskip Finally, ${\rm span}(\mathbf e_1,\mathbf e_2)$ is isometrically equal to $\ell_2(2)$ by Fact \ref{2dim}.
\end{proof}

\section{Schreier families}\label{Schreier}

\subsection{Schreier transfinite sequences} Recall that the classical Schreier family is \[ \Sc_1=\{\emptyset\}\cup \{ \emptyset\neq F\subset\N;\; \vert F\vert\leqslant \min(F)\}.\] 

The so-called ``generalized Schreier families'' have been introduced in \cite{AA}. Let $\omega_1$ be the first uncountable ordinal. We will say that a (transfinite) sequence 
$(\Sc_\alpha)_{0\leqslant  \alpha<\omega_1}$ of families of finite subsets of $\N$ is a \emph{transfinite Schreier sequence} if

\begin{itemize}
    \item $\Sc_0=\{ F\subset\N;\; \vert F\vert\leqslant 1\}$;
    \item $\Sc_{\beta +1}=\bigl\{\bigcup_{i=1}^{n}E_i; \; n \leqslant  E_1<E_2< \dots <E_n \text{ and }E_i\in \Sc_{\beta}\bigr\}$ for every $\beta<\omega_1$;
    \item $\Sc_{\alpha}= \{F \subset \N; \; F \in \Sc_{\alpha_n} \text{ for some }n \leqslant  F\}$ if $\alpha$ is a limit ordinal, where $(\alpha_n)_{n\in\N}$ is a preassigned strictly increasing sequence of ordinals such that $\sup_n\alpha_n=\alpha$. 
\end{itemize}

We will in fact never consider the family $\mathcal S_0$, which is here just for the sake of having a consistent notation.

\smallskip
Note that there is  not just one transfinite Schreier sequence $(\Sc_\alpha)_{\alpha<\omega_1}$, since for limit ordinals $\alpha$ the family $\Sc_\alpha$ depends on the choice of the increasing sequences $(\alpha_n)$ converging to $\alpha$. On the other hand, $\Sc_k$ is uniquely determined for every natural number $k$; and in particular, $\Sc_k$ for $k:= 1$ is indeed the classical Schreier family $\Sc_1$. 

\smallskip
Note also that in the definition of $\Sc_{\beta+1}$, some sets $E_i$ are allowed to be empty: $n\leqslant  E$ is always true if $E=\emptyset$, and similarly $E<E'$ is true if $E$ or $E'$ is $\emptyset$. 

\smallskip
One checks by transfinite induction that for any $1\leqslant  \alpha<\omega_1$, $\Sc_\alpha$ is a spreading combinatorial family such that $\Sc_\alpha^{MAX}=\{\{1\}\}$ and $\{m,n\}\in \Sc_\alpha$ for any $2\leqslant  m<n$. 
 Moreover, denoting by $K^{(\xi)}$ the $\xi$-th Cantor-Bendixson derivate of a compact metric space $K$, it is shown in \cite{AA} that \[ \Sc_\alpha^{(\omega^\alpha)}=\{\emptyset\}.\]

\smallskip These properties hold true for any transfinite Schreier sequence $(\Sc_\alpha)_{\alpha<\omega_1}$, \mbox{\it i.e.} regardless of the choice of the ``approximating sequences'' $(\alpha_n)$. 

\smallskip However, we will need a property that does depend on the choice of the approximating sequences: 

\begin{fact} There exist transfinite Schreier sequences $(\Sc_\alpha)_{\alpha<\omega_1}$ with the following additional property: for every limit ordinal $\alpha<\omega_1$, the approximating sequence $(\alpha_n)$ is made up of  successor ordinals, $\alpha_n=\beta_n+1$, and the sequence 
$(\Sc_{\beta_n})_{n\in\N}$ is increasing. Any transfinite Schreier sequence with that property will be said to be \emph{good}.
\end{fact}
\begin{proof} We construct the transfinite sequence $(\Sc_\alpha)_{\alpha<\omega_1}$ by induction. The only thing we have to prove is the following: if $\alpha<\omega_1$ is a limit ordinal and if $\Sc_\xi$ has been defined for every $\xi<\alpha$, then it is possible to find an approximating sequence $(\alpha_n)$ for $\alpha$ with the required property. 

Let $(\xi_n)_{n\in\N}$ be any strictly increasing sequence of ordinals such that $\sup_{n\in\N} \xi_n =\alpha$. By \cite[Proposition 3.1, (viii)]{Cau}, for every $n\geqslant  2$, one can find an integer $p_n\in\N$ such that $\Sc_{\xi_{n-1}}\subset \Sc_{\xi_{n}+p_n}$. So if we set 
$\beta_1:=\xi_1$ and $\beta_{n}:=\xi_n+p_n+\cdots +p_2$ for $n\geqslant  2$, then the sequence $(\Sc_{\beta_n})_{n\in\N}$ is increasing; so we may take $\alpha_n:=\beta_n+1$.
\end{proof}

\medskip We now state the precise form of part (2) of Theorem \ref{main}.

\begin{proposition}\label{restated} Let $1<p<\infty$, and let $(\Sc_\alpha)_{\alpha<\omega_1}$ be a transfinite Schreier sequence.
\begin{itemize}
\item[\rm (a)] If $p\neq 2$ then, for every $1\leqslant  \alpha<\omega_1$, all isometries of $X_{\Sc_\alpha, p}^*$ are diagonal.
\item[\rm (b)] If $p=2$ then: for every successor ordinal $1\leqslant  \alpha<\omega_1$, all isometries of $X_{\Sc_\alpha, 2}$ are diagonal; and  the same holds true for every $1\leqslant  \alpha<\omega_1$ if $(\Sc_\alpha)_{\alpha<\omega_1}$ is good.
\end{itemize}
\end{proposition}
\begin{remark} In the complex case, we can show that in fact (a) holds true for $p=2$ as well; see Corollary \ref{Scfin}.
\end{remark}

In what follows, we fix a transfinite Schreier sequence $(\Sc_\alpha)_{\alpha<\omega_1}$, determined by the choice of an approximating sequence $(\alpha_n)$ for every limit ordinal $\alpha<\omega_1$.

\subsection{The case $p\neq 2$} Let $1<p\neq 2<\infty$. By the already proved part (1) of Theorem~\ref{main}, we know that any isometry of  $X_{\Sc_\alpha, p}^*$, $\alpha<\omega_1$  is an $\Sc_\alpha\,$-$\,$compatible signed permutation of $(e_i^*)$. So, to prove that the isometries of  $X_{\Sc_\alpha, p}^*$, for any $1\leqslant  \alpha<\omega_1$, are diagonal, it is enough to show that the only permutation $\pi: \N\to\N$ such that $\pi(\Sc_\alpha)=\Sc_\alpha$ is the identity. This fact was hinted at in \cite{BP}, but we couldn't locate a complete proof.

\smallskip Given any compact metric space $K$, we can define a function $\mathrm{rk}_K : K\to \omega_1$ which associates to $x \in K$ the largest ordinal $\beta$ for which $x \in K^{(\beta)}$. Now, if $\F$ is any compact family of subsets of $\N$ and if $\pi:\N\to\N$ is a permutation, then the map $F\mapsto \pi(F)$ is a homeomorphism from $\F$ onto $\pi(\F)$, and hence ${\rm rk}_\F (F)={\rm rk}_{\pi(\F)}(\pi(F))$ for every $F\in\F$. In particular, if $\F$ contains all singletons and $\pi(\F)=\F$, then ${\rm rk}_\F(\{\pi(n)\})={\rm rk}_\F(\{ n\})$ for all $n\in\N$. As observed in \cite{BP}, it follows that if the map $n\mapsto {\rm rk}_\F(\{ n\})$ is injective, then $\pi$ has to be the identity. Therefore, it is enough to prove the following lemma.

\begin{lemma}\label{rk}
For any $1\leqslant  \alpha <\omega_1$, the map $n\mapsto \mathrm{rk}_{\Sc_\alpha}(\{n\})$ is strictly increasing.  
\end{lemma}
\begin{proof} Recall that if $F$ is a subset of $\N$ and $k\in\N$, then ``$k<F$'' means that $k< i$ for all $i\in F$. In particular, this is true if $F=\emptyset$.
\begin{fact}\label{f1} Let $F$ be a finite subset of $\N$, and let $m,n,p\in\N$. 
    If $m<n<p<F$ and $\{m\}\cup F \in \Sc_\alpha$ where $1\leqslant \alpha<\omega_1$, then $\{n,p\}\cup F \in \Sc_\alpha$.
\end{fact}
\begin{proof}[Proof of Fact \ref{f1}]
    We prove the statement by transfinite induction on $\alpha$.
    
    For $\alpha=1$, since $\{m\}\cup F \in \Sc_1$, it follows that $\vert F\vert +1= \vert \{m\} \cup F\vert\leqslant  \min(\{m\} \cup F)= m$. Hence, $\vert \{n,p\}\cup F\vert =\vert F\vert +2\leqslant  m+1\leqslant  n=\min(\{n,p\}\cup F)$, and so $\{n,p\}\in \Sc_1$.

    Assume that the statement holds true for all $1\leqslant  \beta<\alpha$, and let us prove it for $\alpha$. 

    If $\alpha$ is a successor ordinal, then $\alpha= \beta+1$, for some $\beta$. Since $\{m\}\cup F \in \Sc_\alpha$, it follows by definition that $\{m\}\cup F = G_1 \cup \dots \cup G_r$ where $G_1<\dots <G_r $, $G_i \in \Sc_\beta$ and $r\leqslant  \min G_1=m$. We have $G_1=\{m\} \cup H$ where $H>p>n>m$ because $F>p>n>m$, hence $\{n,p\}\cup H\in \Sc_\beta$ by the inductive hypothesis. Therefore, $\{n,p\}\cup F=(\{n,p\}\cup H)\cup G_2 \cup \dots \cup G_r \in \Sc_\alpha$, since $r\leqslant  \min(\{n,p\}\cup H)=n$. 

   If $\alpha$ is a limit ordinal then, since $\{m\}\cup F \in \Sc_\alpha$, it follows by definition that $\{m\}\cup F  \in \Sc_{\alpha_k}$ for some $k \leqslant  \min(\{m\}\cup F)=m$. Hence, $\{n,p\}\cup F\in \Sc_{\alpha_k}$ by the inductive hypothesis, and so $\{n,p\}\cup F\in \Sc_{\alpha}$ since $k \leqslant  n = \min(\{n,p\}\cup F)$.
    \end{proof}
    \begin{fact}\label{f2} Let $1\leqslant \alpha<\omega_1$, and let $m<n<p\in\N$. If $F$ is a finite subset of $\N$ such that
 $m<n<p<F$ and $\{m\}\cup F \in \Sc_\alpha^{(\mu)}$ where $0\leqslant  \mu<\omega_1$, then $\{n,p\}\cup F\in \Sc_\alpha^{(\mu)}$.
    \end{fact}
    \begin{proof}[Proof of Fact \ref{f2}] 
         We prove the statement by transfinite induction on $\mu$.   For $\mu=0$, the statement holds true by Fact \ref{f1}.  Assume that it holds true for all  $\xi<\mu$, and let us prove it for $\mu$. 

           If $\mu$ is a limit ordinal, then the statement is trivially true. 

           Assume that $\mu$ is a successor ordinal, $\mu=\xi +1$. Since $\{m\}\cup F \in \Sc_\alpha^{(\mu)}$, we can find a sequence $(G_k)\subseteq \Sc_\alpha^{(\xi)}$ such that $G_k \to \{m\}\cup F$ and $G_k\neq \{m\}\cup F$. For $k$ large enough, we may write $G_k= \{m\}\cup F \cup H_k$, where $H_k>F$ and $H_k\neq \emptyset$. Moreover, $H_k \to \emptyset$. By the inductive hypothesis, we have $\{n,p\}\cup F\cup H_k\in \Sc_\alpha^{(\xi)}$. Hence, $\{n,p\}\cup F \in \Sc_\alpha^{(\mu)}$ since $\{n,p\}\cup F \cup H_k \to \{n,p\}\cup F$ and $\{n,p\}\cup F \cup H_k \neq \{n,p\}\cup F$.
\end{proof}
\begin{fact}\label{f3}
    If $m<n $ and $\{m\}\in \Sc_\alpha^{(\mu)}$ where $0\leqslant \mu<\omega_1$, then $\{n\}\in \Sc_\alpha^{(\mu+1)}$.
\end{fact}
\begin{proof}
    By Fact \ref{f2}, we have $\{n,p\}\in \Sc_\alpha^{(\mu)}$ for all $p>n$; and $\{n,p\}\to \{n\}$ when $p\to \infty$.
\end{proof}
Fact \ref{f3} implies that if $m<n$ then $\mathrm{rk}_{\Sc_\alpha}(\{m\})<\mathrm{rk}_{\Sc_\alpha}(\{n\})$, which finishes the proof of Lemma \ref{rk}. 
\end{proof}

\subsection{The case $p=2$} When $p=2$, we have not proved that any isometry of $X_{\Sc_\alpha, 2}^*$ is a signed permutation of $(e_i^*)$, so we cannot use the above argument. Instead, we will directly show that the isometries of $X_{\Sc_\alpha,2}$ are diagonal for every $1\leqslant \alpha<\omega_1$ by adapting the proof given in \cite{ABC} that 
the isometries of the combinatorial Banach space $X_{\Sc_n}$ are diagonal when $n$ is a natural number. To handle the limit ordinals, we will need to assume that the transfinite Schreier sequence $(\Sc_\alpha)_{\alpha<\omega_1}$ is good, but we suspect this is not necessary. 

\smallskip The proof we give works for both the real and the complex case. However, in the complex case a different approach is possible, and it gives a better result; see Corollary~\ref{Scfin}.

\subsubsection{Preliminary facts} The following theorem from \cite{ABC} gives a description of the extreme points of the unit ball of $X_{\Sc_\alpha,p}$, $1\leqslant  \alpha< \omega_1$, $1<p<\infty$. In \cite{ABC}, the authors consider \emph{real} spaces only, but minor modifications of the proof give the result in the complex case as well. 
\begin{theorem}\label{extSc}
   Let $\F:=\Sc_\alpha$, $1\leqslant  \alpha< \omega_1$, and assume that the sequence $(\Sc_{\alpha_n})$ is increasing if $\alpha$ is a limit ordinal. Let $1<p<\infty$, and let $x \in {X_{\F,p}}$ with $\Vert x\Vert=1$. Then $x \in \mathrm{Ext}(B_{X_{\F,p}}) $ if and only if 
   \begin{itemize}
   \item[{\tiny$\bullet$}] $x\in c_{00}$, 
   \item[{\tiny$\bullet$}] the family $\A_x := \{F\in \F; \; \sum_{i \in F}\vert x_i\vert^p=1\}$ contains a set $F\notin \F^{MAX}$, 
   \item[{\tiny$\bullet$}] for all $i \leqslant  \max\mathrm{supp}(x) $, there exists $F \in \A_x$ such that $i \in F$.
   \end{itemize}
\end{theorem}

\medskip We will also need two lemmas concerning the Schreier families $\Sc_\alpha$.

\smallskip
\begin{lemma}\label{max}
    If $1\leqslant  \alpha <\omega_1$ and $G \in \Sc_\alpha \minus \Sc_\alpha^{MAX}$, then $G \cup \{l\}\in \Sc_\alpha$, for all $l > G$.
\end{lemma}
\begin{proof} This is \cite[Lemma 3.1]{G}. 
\end{proof}

\smallskip
\begin{lemma}\label{maxsucc}
 Let $\alpha<\omega_1$ be a successor ordinal, $\alpha=\beta+1$. If $G \in \Sc_\alpha^{MAX}$, then $G= \bigcup_{i=1}^m G_i$ where $G_1<\dots<G_m \in \Sc_\beta^{MAX}$ and $m=\min G_1$.    
\end{lemma}
\begin{proof} This is proved in \cite[Lemma 3.8]{GL} when $\alpha$ is a natural number. 
Let $G \in \Sc_\alpha^{MAX}$. By definition, $G=\bigcup_{i=1}^m G_i$ where $G_1< \dots < G_m \in \Sc_\beta$ and $m \leqslant  G_1$. \\
First, we observe that $m= \min G_1$. Indeed, otherwise  $m+1\leqslant  G_1$, so, taking any $l> \max G$, we have that $\bigcup_{i=1}^m G_i \cup \{l\} \in \Sc_\alpha$, \textit{i.e.} $G\cup \{l\}\in \Sc_\alpha$, which contradicts the fact that $G \in \Sc_\alpha^{MAX}$. \\
Now, we show that $G_i \in \Sc_\beta^{MAX}$ for each $i \in \{1,\dots, m\}$. Towards a contradiction, suppose that there exists $i_0 \in \{1,\dots, m\}$ such that $G_{i_0} \in \Sc_\beta \minus \Sc_\beta^{MAX}$.\\ 
If $i_0 =m$, then $G_m \notin \Sc_\beta ^{MAX}$, and Lemma \ref{max} implies that for every $l > \max G_m=\max G$, we have $G_m \cup \{l\}\in \Sc_\beta$. Hence, $\bigcup_{i=1}^{m-1}G_i \cup (G_m\cup \{l\})\in \Sc_\alpha$, which is a contradiction because $G\in \Sc_\alpha^{MAX}$.\\
If $i_0< m$, then $G'_{i_0}:= G_{i_0}\cup \{\min G_{i_0+1}\} \in \Sc_\beta$ by Lemma \ref{max}. Similarly, Lemma \ref{max} implies that $G'_{i} := (G_{i}\minus \{\min G_{i}\})\cup \{\min G_{i+1}\} \in \Sc_\beta$ for all $i_0<i<m$ , and that $G'_m := (G_{m}\minus \{\min G_{m}\})\cup \{l\} \in \Sc_\beta$, for any $l> \max G$. So taking $l>\max G$, we see that $G\cup\{ l\}=\bigcup_{1 \leqslant  i<i_0}G_i \cup \bigcup_{i=i_0}^m G'_i \in \Sc_\alpha$, because $m=\min G_1$ and $G_i, G'_i \in \Sc_\beta$, which contradicts that fact that $G \in S_\alpha^{MAX}$. 
\end{proof}

\subsubsection{Successor ordinals} In this section, we fix a successor ordinal $\alpha=\beta+1$, and we show that every isometry of $X_{\Sc_\alpha, 2}$ is diagonal.

\medskip Let $T:X_{\Sc_\alpha,2}\to X_{\Sc_\alpha,2}$ be an isometry. For every $i\in\N$, we set 
\[ Te_i=: f_i\qquad{\rm and}\qquad T^{-1}e_i=: d_i.\]

\medskip
 \begin{step1}
For every $i \in \N$, the vectors $d_i$ and $f_i $ belong to $c_{00}$.
 \end{step1}
 \begin{proof} 
Theorem \ref{extSc} implies that $e_1 \pm e_i \in \mathrm{Ext}(B_{X_{\Sc_\alpha,2}})$ for all $i \in \N \minus \{1\}$. Since $T$ preserves extreme points, $f_1+f_i, f_1-f_i, d_1+d_i, d_1-d_i \in  \mathrm{Ext}(B_{X_{\Sc_\alpha,2}}) \subset c_{00}$ (by Theorem \ref{extSc}). Hence, $f_i, d_i \in c_{00}$ for all $i \in \N$.
\end{proof}
\begin{step2}
We have $f_1=\omega_1 e_1$ for some sign $\omega_1$.
\end{step2}
\begin{proof}
Since $\Vert f_1\Vert=\Vert e_1\Vert=1$, it is enough to show that ${\rm supp}(f_1)=\{ 1\}$. 
Towards a contradiction, suppose that $k:= \max \mathrm{supp} (f_1) >1$ (this maximum exists because $f_1 \in c_{00}$ by Step 1). \\
Let $l >k$. Since $\{k,l\} \in \Sc_{\alpha}$ and $l \notin \mathrm{supp}f_1$, we have for any sign $\omega$:
\[\Vert e_1 +\omega d_l\Vert ^2 = \Vert f_1 +\omega e_l \Vert ^2 \geqslant    \vert f_1(k )\vert^2 +1>1.  \]
So, for any sign $\omega$, one can find $F_\omega \in \Sc_\alpha$ such that $\sum_{i \in F_\omega} \vert e_1(i)+ \omega d_l(i)\vert^2 >1$. Since $\Vert d_l\Vert=1$, we must have $1 \in F_\omega$, and hence $F_\omega =\{1\}$ because $\{ 1\}\in\Sc_\alpha^{MAX}$. So we have $\vert 1+\omega d_l(1)\vert^2>1$ for every sign $\omega$, which contradicts the fact that $\vert d_l(1)\vert \leqslant  1$.
\end{proof}
\begin{step3}
For every $k\in \N\minus \{1\}$, we have $d_k(1)=f_k(1)=0$. 
\end{step3}
\begin{proof} By
Step 2, we know that $f_1=\omega_1 e_1$. Since $\{ 1, k\}\notin\Sc_\alpha$, we then have for any sign~$\varepsilon$: 
\[  1= \Vert e_1+\varepsilon e_k \Vert = \Vert \omega_1 e_1 + \varepsilon f_k \Vert \geqslant    \vert \omega_1+ \varepsilon f_k(1)\vert.\]
Taking $\varepsilon$ such that $\vert \omega_1+ \varepsilon f_k(1)\vert =1+\vert f_k(1)\vert$, it follows that  $f_k(1)=0$. One shows in the same way that $d_k(1)=0$. 
\end{proof}
\begin{step4}
 Let $n,m \in \N\minus \{1\}$. If $f_n(m)\neq 0$, then $d_m(n)\neq 0$. 
\end{step4}
\begin{proof}
By Theorem \ref{extSc}, we know that $\varepsilon e_1+e_n\in \mathrm{Ext}(B_{X_{\Sc_\alpha,2}})$ for any sign $\varepsilon $. Since $e_1+f_n= \overline{\omega_1}\, f_1+f_n=T(\overline{\omega_1}\,e_1+e_n)$, it follows that $e_1+ f_n \in \mathrm{Ext}(B_{X_{\Sc_\alpha,2}})$. Hence, by Theorem~\ref{extSc} again and since $m \in \mathrm{supp}f_n$,  there exists $F \in \Sc_\alpha$ such that $\sum_{k \in F}\vert e_1(k)+f_n(k)\vert^2=1$ and $m \in F$. Then $F\neq\{ 1\}$ since $m\neq 1$, so $1\notin F$ because $\{ 1\}\in \Sc_\alpha^{MAX}$; and consequently  $\sum_{k \in F}\vert f_n(k)\vert ^2=1.$

\smallskip
Let $\varepsilon$ be a sign such that ${\rm Re}(\overline\varepsilon {f_n(m)})>0$. We have 
\[ \begin{aligned}
    \Vert f_n + \varepsilon e_m \Vert ^2 & \geqslant    \sum_{k \in F} \vert f_n(k)+ \varepsilon e_m(k)\vert ^2 \\
    &= \sum_{k \in F \minus \{m\}} \vert f_n(k)\vert ^2 + \vert f_n(m)+\varepsilon\vert^2\\
    &= \sum_{k \in F \minus \{m\}} \vert f_n(k)\vert ^2 + \vert f_n(m)\vert^2+1+2  \,{\rm Re}(\overline\varepsilon {f_n(m)}) \\
    &=\sum_{k\in F}\vert f_n(k) \vert^2+1 +2 \,{\rm Re}(\overline\varepsilon {f_n(m)}) \\
    &=2+2  \,{\rm Re}(\overline\varepsilon {f_n(m)}) >2. 
\end{aligned}\]

Hence $\Vert f_n + \varepsilon e_m\Vert ^2 >2$ and so \[ \Vert e_n + \varepsilon d_m \Vert^2>2.\]

\smallskip
Now, towards a contradiction, suppose that $n \notin \mathrm{supp}(d_m)$. Then, for any $G \in \Sc_\alpha$, we have
\[ \sum_{k\in G}\vert e_n(k)+ \varepsilon d_m(k)\vert^2=1+ \sum_{k \in G\minus \{n\}}\vert d_m(k)\vert^2\leqslant  2\qquad\hbox{if $n\in G$}\] 
and 
\[ \sum_{k \in G}\vert e_n(k)+ \varepsilon d_m(k)\vert^2=\sum_{k\in G}\vert d_m(k)\vert ^2 \leqslant  1\qquad\hbox{if $n\notin G$}.\]

Therefore, $\Vert e_n + \varepsilon d_m \Vert^2\leqslant  2$, which is the required contradiction. 
\end{proof}
\begin{step5}
For every $k \geqslant  2$, there exists $n >k$ such that $k<\mathrm{supp}(f_n)$. 
\end{step5}
\begin{proof} Let $k \geqslant  2 $. Towards a contradiction, suppose that for all $n >k$, $\mathrm{supp}f_n \cap \{2, \dots ,k\}\neq\emptyset $ (recall that $f_n(1)=0$ by Step 3). This implies that there exist $a \in \{2, \dots ,k\}$ and infinitely many $n>k$ such that $f_n(a)\neq 0$. It follows from Step 4 that $d_a(n) \neq 0$ for infinitely many $n$, which is a contradiction since $d_a \in c_{00}$ by Step 1. 
\end{proof}

\begin{step6}
For every $k \geqslant  2$, $\mathrm{supp}(d_k) \leqslant  k$. 
\end{step6}
\begin{proof} 
Let $k_1:= \max\mathrm{supp}(d_k) +1$ (hence $k_1 >2 $ because $d_k(1)=0$ by Step 3). By Step 5, we can fix $n_1 > k_1$ such that $\mathrm{supp}f_{n_1}>k_1$. Next, let $k_2 := \max\mathrm{supp}(f_{n_1})$ and fix $n_2 > \max \{k_2,n_1\}$ such that $\mathrm{supp}(f_{n_2})>\max \{k_2,n_1\}$.  
Continuing in this way, we obtain two strictly increasing sequences $(k_j)_{j=1}^\infty$ and $(n_j)_{j=1}^\infty$ such that for every $j \geqslant   2$, $k_j = \max\mathrm{supp}(f_{n_{j-1}})$, $n_j> \max \{k_j,n_{j-1}\}$ and $\mathrm{supp}(f_{n_j})> \max \{k_j,n_{j-1}\}$. 

\smallskip Recall that $\alpha$ is a successor ordinal,  $\alpha=\beta+1$.

\smallskip
Since $k_1>2$ (hence $\{k_1 \}\in \Sc_\beta\minus \Sc_\beta^{MAX}$), Lemma \ref{max}, together with the fact that $\Sc_\beta$ is a compact family of finite subsets of $\N$, implies that there exists $r_1 >1$ such that $F_1:= \{k_j; \; j=1, \dots, r_1\}\in \Sc_\beta^{MAX}$. 
Let $F'_1:= \{n_j; \; j=1, \dots, r_1\}$. Note that $F'_1 \in \Sc_\beta$ because the family $\Sc_\beta$ is spreading. \\
In the same way, for every $i \in \{2, \dots , k+1\}$, there exists $r_i>r_{i-1}+1$ such that $F_i:= \{k_j; \; j=r_{i-1}+1, \dots, r_i\}\in \Sc_\beta^{MAX}$,  and then $F'_i  := \{n_j; \; j=r_{i-1}+1, \dots, r_i\}\in \Sc_\beta$. \\
So, we have $F_1< \dots < F_k < F_{k+1}$, $F'_1< \dots < F'_k < F'_{k+1}$ and $\vert F_i\vert = \vert F'_i\vert$, for all $i \in \{1,\dots, k+1\}$.

\smallskip Now, towards a contradiction, assume that $\max \mathrm{supp}(d_k) \geqslant   k+1$.  

\smallskip 
Since $T^{-1}(e_k + \sum_{i=1}^k\sum_{n \in F'_i} f_n)= d_k+ \sum_{i=1}^k\sum_{n \in F'_i} e_n$, we have
\[ \Bigl\Vert d_k+ \sum_{i=1}^k\sum_{n \in F'_i} e_n \Bigr\Vert = \Bigl\Vert e_k + \sum_{i=1}^k\sum_{n \in F'_i} f_n \Bigr\Vert.  \]
So, we will get the required contradiction if we are able to show that on the one hand
\begin{align}\label{1}
    \Bigl\Vert d_k+ \sum_{i=1}^k\sum_{n \in F'_i} e_n \Bigr\Vert ^2 > \sum_{i=1}^k \vert F'_i \vert,
\end{align}
and on the other hand
\begin{align}\label{2}
   \Bigl\Vert e_k + \sum_{i=1}^k\sum_{n \in F'_i} f_n \Bigr\Vert ^2 \leqslant   \sum_{i=1}^k \vert F'_i \vert. 
\end{align}

Let us first prove (\ref{1}). \\
Since $k_1=\max \mathrm{supp}(d_k)+1  > k+1$, we have $\min (\{k+1\}\cup \bigcup_{i=1}^k F_i)=k+1$ and hence $\{k+1\}\cup \bigcup_{i=1}^k F_i \in \Sc_\alpha$. So, choosing any 
 $m \in \mathrm{supp}(d_k)$ such that $m \geqslant   k+1$, we see that $\{m\}\cup \bigcup_{i=1}^k F'_i \in \Sc_{\alpha}$ because  $\Sc_\alpha$ is spreading. 
Since $d_k(m)\neq 0$  and $m \notin \bigcup_{i=1}^k F'_i$ (because $\min(\bigcup_{i=1}^k F'_i)=n_1>k_1>m$), it follows that 
\[  \Bigl\Vert d_k+ \sum_{i=1}^k\sum_{n \in F'_i} e_n \Bigr\Vert ^2 \geqslant   \vert d_k(m)\vert ^2 + \sum_{i=1}^k\vert F'_i\vert > \sum_{i=1}^k\vert F'_i\vert.   \]

Let us now prove (\ref{2}). \\
Fix $G \in \Sc_\alpha$, and assume without loss of generality that $G \in \Sc_\alpha^{MAX}$. By Lemma \ref{maxsucc}, we may write $G = \bigcup_{i=1}^m G_i$ where $G_1< \dots< G_m \in \Sc_\beta^{MAX}$ and $\min G_1 =m$. \\
Note that if $k \notin G$ or $G \cap \mathrm{supp}(f_n) = \emptyset $ for some $n \in \bigcup_{i=1}^k F'_i$, then 
\[ \sum_{t \in G} \Bigl\vert e_k(t)+\sum_{i=1}^k\sum_{n \in F'_i} f_n(t)\Bigr\vert ^2 \leqslant  \sum_{i=1}^k \vert F'_i\vert, \]
because $k \notin \bigcup_{i=1}^k\bigcup_{j\in F'_i} \mathrm{supp}(f_j)$ and the $f_n$,  
$ n\in \bigcup_{i=1}^k F'_i$ are disjointly supported.\\
So, let us assume that  
\begin{align}\label{assumption} k \in G \qquad \hbox{and} \qquad G \cap \mathrm{supp}(f_n) \neq \emptyset \quad\hbox{for all} \; n \in \bigcup_{i =1}^k F'_i. \end{align}
Since $k \in G $, it follows that $\min G \leqslant  k$ and so $m=\min G_1 \leqslant  k$. 

\smallskip\noindent We claim that there exists $n^1\in F'_1=\{ n_1, \dots ,n_{r_1}\}$ such that $G_1 \cap \mathrm{supp}(f_{n^1})=\emptyset$. Indeed, otherwise we may choose $s_i\in G_1 \cap \mathrm{supp}(f_{n_i})$ for $i=1,\dots ,r_1$. Note that $k_i < s_i \leqslant  k_{i+1}$ for all $i\in\{ 1,\dots ,r_1\}$. Since $\{k_1, \dots, k_{r_1}, k_{r_1+1}\}\notin \Sc_\beta$ (because $\{k_1, \dots, k_{r_1}\}=F_1\in \Sc_\beta^{MAX}$) and since $\Sc_\beta$ is spreading, it follows that $\{\min G_1, s_1, \dots,s_{r_1} \} \notin \Sc_\beta$. However, $\{\min G_1, s_1, \dots,s_{r_1} \} \subset G_1 \in \Sc_\beta$, hence  $\{\min G_1, s_1, \dots,s_{r_1} \} \in \Sc_\beta$, which is a contradiction. This proves our claim. 
 Note that since $G \cap \mathrm{supp}f_{n^1}\neq \emptyset$ by (\ref{assumption}), we must have $\min G_2 \leqslant  \max\mathrm{supp}f_{n^1}$. Hence $\min G_2\leqslant    \max\mathrm{supp}f_{n_{r_1}}=k_{r_1+1}= \min F_2$.\\
One shows in the same way that for every $i \in \{1,\dots , m\}$, there exists $n^i \in F'_i$ such that $G_i \cap \mathrm{supp}f_{n^i}=\emptyset$, and $\min G_i\leqslant  \min F_i$. In particular, $\min G_m\leqslant  \min F_m$ and 
\begin{align}\label{res1}
\exists \; n^m\in F'_m \qquad \hbox{such that} \qquad G_m \cap \mathrm{supp}(f_{n^m})= \emptyset.
\end{align}
Since $\min F_m<\min \mathrm{supp}(f_{n^m})$ (by construction), it follows that $\min G_m <\min \mathrm{supp}(f_{n^m})$ and this implies that 
\begin{align}\label{res2}
    \forall \; 1 \leqslant  i <m  \qquad  G_i \cap \mathrm{supp}(f_{n^m})=\emptyset 
\end{align}
because $G_i<G_m$. 
So we have found $n^m \in F'_m \subset \bigcup_{i =1}^k F'_i$ (because $m\leqslant  k)$ such that $G \cap \mathrm{supp}f_{n^m}=\emptyset $ which contradicts (\ref{assumption}). \\
Therefore, there is in fact no $G \in \Sc_\alpha^{MAX}$ that verifies (\ref{assumption}). So, we have proved (\ref{2}), and altogether ${\rm supp}(d_k)\leqslant  k$.
  \end{proof}

\begin{step7} We have 
  $f_2=\omega_2 e_2$ for some sign $\omega_2$.     
\end{step7}
\begin{proof} By Step 6 and Step 3, we see that  $\mathrm{supp}(d_2)=\{2\}$.
\end{proof}

\begin{step8}
  For every $k\in \N \minus \{1\}$, we have $f_k=\omega_k e_k$ for some sign $\omega_k$.   
\end{step8}
\begin{proof} We prove that by induction. The statement holds true for $k=2$ by Step 7. Fix $k\geqslant   3$, and assume that it holds true for every $2\leqslant  j<k$; we will prove it for $k$, \textit{i.e.} we show that ${\rm supp}(d_k)=\{ k\}$.

\smallskip By Step 3 and Step 6, we already know that $d_k(1)=0$ and ${\rm supp}(d_k)\leqslant  k$. Towards a contradiction, assume that there exists $2\leqslant  j<k$ such that $d_k(j)\neq 0$.

\smallskip 
Since $j \in \mathrm{supp}(d_k)$ and $e_1+d_k =T^{-1}({\omega_1} e_1+e_k)\in \mathrm{Ext}(B_{X_{\Sc_\alpha,2}})$, Theorem \ref{extSc} implies that there exists $F \in \Sc_\alpha$ such that $j \in F$ and $\sum_{i \in F}\vert e_1(i)+ d_k(i)\vert^2=1$. Then $F \neq \{1\}$; hence $1\notin F$ because $\{ 1\}\in\Sc_\alpha^{MAX}$, so that $\sum_{i \in F}\vert  d_k(i)\vert^2=1$.\\ 
Let $\varepsilon$ be a sign such that $\mathrm{Re}(\varepsilon d_k(j))>0$. Since $\{j,k\}\in \Sc_\alpha$ (because $k>j\geqslant   2$) and $T^{-1}({\omega_j} e_j)=e_j$ (by the inductive hypothesis), we have $\Vert e_j+\varepsilon d_k\Vert^2=\Vert \omega_j e_j +\varepsilon e_k\Vert^2=2$. \\
However, \[\begin{aligned}
    \Vert e_j+\varepsilon d_k\Vert^2 &\geqslant   \sum_{i \in F} \vert e_j(i)+ \varepsilon d_k(i)\vert^2\\
    &= \vert 1+ \varepsilon d_k(j) \vert ^2 + \sum_{i \in F\minus \{j\}} \vert  d_k(i)\vert^2\\
    &=1+ \vert d_k(j)\vert ^2+ 2 \,{\rm Re}(\varepsilon d_k(j))+ \sum_{i \in F\minus \{j\}} \vert  d_k(i)\vert^2 \\
    &=1+ \sum_{i \in F} \vert  d_k(i)\vert^2 + 2 \,{\rm Re}(\varepsilon d_k(j)) >2,
    \end{aligned}\]
which is the required contradiction. 
\end{proof}  

\subsubsection{Limit ordinals} In this section, we fix a limit ordinal $\alpha<\omega_1$, and we prove that any isometry of $T: X_{\Sc_\alpha, 2}\to X_{\Sc_\alpha, 2}$ is diagonal.

\medskip Recall that, in accordance with the statement of Proposition \ref{restated}, we assume that the approximating sequence $(\alpha_n)$ for $\alpha$ has been chosen in such a way that the $\alpha_n$ are successor ordinals, $\alpha_n=\beta_n+1$, and that the sequence 
$(\Sc_{\beta_n})_{n\in\N}$ is increasing.

\medskip
Looking back at the proof done in the successor case, we see that the only place where we used that $\alpha$ was a successor ordinal is in Step 6. So, if we can prove Step 6 in our present setting, we get the desired result. In what follows, we keep the notation used in the proof of the successor case; so we have to show that $\mathrm{supp}(d_k) \leqslant  k$, for every $k\geqslant   2$.

\medskip
First, we construct two strictly increasing sequences $(k_j)_{j=1}^\infty$ and $(n_j)_{j=1}^\infty$ in the same way as in the proof of Step 6.

\smallskip Then, we continue as in Step 6 but replacing $\beta$ with $\beta_k$. In other words, since $k_1>2$, Lemma \ref{max} implies that there exists $r_1 >1$ such that $F_1:= \{k_j; \; j=1, \dots, r_1\}\in \Sc_{\beta_k}^{MAX}$, and we set $F_1':= \{n_j; \; j=1, \dots, r_1\}$. Note that $F_1' \in \Sc_{\beta_k}$ because $\Sc_{\beta_k}$ is spreading. \\
In the same way, for every $i \in \{2, \dots , k+1\}$, there exists $r_i>r_{i-1}+1$ such that $F_i:= \{k_j; \; j=r_{i-1}+1, \dots, r_i\}\in \Sc_{\beta_k}^{MAX}$, and then $F'_i  := \{n_j; \; j=r_{i-1}+1, \dots, r_i\}\in \Sc_{\beta_k}$. 

\smallskip
Now, we show, that if $\max \mathrm{supp}(d_k) \geqslant   k+1$, then we have (\ref{1}) and (\ref{2}) and hence a contradiction (see the proof of Step 6). 

Let us first prove (\ref{1}). Since $k_1 > k+1$ (because we are assuming that $\max \mathrm{supp}(d_k) >k$), we have $\min (\{k+1\}\cup \bigcup_{i=1}^k F_i)=k+1$ and hence $\{k+1\}\cup \bigcup_{i=1}^k F_i \in \Sc_{\beta_k+1}=\Sc_{\alpha_k}$. Therefore,  $\{k+1\}\cup \bigcup_{i=1}^k F_i \in \Sc_\alpha$, since $ k+1\leqslant   \{k+1\}\cup \bigcup_{i=1}^k F_i$.  So, if we choose $m \in \mathrm{supp}(d_k)$  such that $m \geqslant   k+1$, then $\{m\}\cup \bigcup_{i=1}^k F'_i \in \Sc_{\alpha}$ because  $\Sc_\alpha$ is spreading. 
Since $d_k(m)\neq 0$ and $m \notin \bigcup_{i=1}^k F'_i$ (because $\min(\bigcup_{i=1}^k F'_i)=n_1>k_1>m$), it follows that 
\[  \Bigl\Vert d_k+ \sum_{i=1}^k\sum_{n \in F'_i} e_n \Bigr\Vert ^2 \geqslant   \vert d_k(m)\vert ^2 + \sum_{i=1}^k\vert F'_i\vert > \sum_{i=1}^k\vert F'_i\vert.   \] 
So, we have shown (\ref{1}). 

Finally, let us prove (\ref{2}). \\
Fix a set $G \in \Sc_\alpha^{MAX}$. 
If $k \notin G$ or $G \cap \mathrm{supp}f_n = \emptyset $ for some $n \in \bigcup_{i=1}^k F'_i$, then 
\[ \sum_{t \in G} \Bigl\vert e_k(t)+\sum_{i=1}^k\sum_{n \in F'_i} f_n(t)\Bigr\vert ^2 \leqslant  \sum_{i=1}^k \vert F'_i\vert. \]
So, let us assume that $G$ verifies (\ref{assumption}).\\
Since $k\in G$, there exists $n \leqslant  \min G \leqslant  k$ such that $G \in \Sc_{\alpha_n}$.  
Note that $G \in \Sc_{\alpha_n}^{MAX}$. Indeed, if $G \in \Sc_{\alpha_n}\minus \Sc_{\alpha_n}^{MAX}$, then Lemma \ref{max} implies that  $G \cup \{l\} \in \Sc_{\alpha_n}$, for each $l>G$. Since $ n \leqslant  \min G= \min(G\cup\{l\})$, it follows that $G\cup\{l\} \in \Sc_\alpha$, which contradicts the fact that $G \in \Sc_\alpha^{MAX}$.\\
Since $G \in \Sc_{\alpha_n}^{MAX}$, where $\alpha_n = \beta_n+1$, Lemma \ref{maxsucc} implies that $G = \bigcup_{i=1}^m G_i$ where $G_1< \dots< G_m \in \Sc_{\beta_n}^{MAX}$ and $\min G_1 =m$ (hence $m\leqslant  k$ because $k\in G$).\\
Assume that  $G_1 \cap \mathrm{supp}f_n \neq \emptyset$, for all $n \in F'_1= \{n_1, \dots ,n_{r_1}\}$. Then, for $i=1,\dots ,r_1$, one can find $s_i \in G_1$ such that $k_i < s_i \leqslant  k_{i+1}$. 
Since $ \{k_1, \dots, k_{r_1}\}=F_1\in \Sc_{\beta_k}^{MAX}$, it follows that $\{k_1, \dots, k_{r_1}, k_{r_1+1}\}\notin \Sc_{\beta_k}$, which implies that $\{\min G_1, s_1, \dots,s_{r_1} \} \notin \Sc_{\beta_k}$ because $\Sc_{\beta_k}$ is spreading. Since $\Sc_{\beta_n}\subset \Sc_{\beta_k}$ (because $n \leqslant  k$), we get $\{\min G_1, s_1, \dots,s_{r_1} \} \notin \Sc_{\beta_n}$. However, since $\{\min G_1, s_1, \dots,s_{r_1} \} \subset G_1 \in \Sc_{\beta_n}$, we do have  $\{\min G_1, s_1, \dots,s_{r_1}\} \in \Sc_{\beta_n}$, which is a contradiction. \\
Therefore, there exists $n^1\in F'_1$ such that $G_1 \cap \mathrm{supp}f_{n^1}=\emptyset$; and as above this implies that $\min G_2 \leqslant   \min F_2$.\\
Continuing in this way, we see that for every $j \in \{1,\dots , m\}$, we have $\min G_j\leqslant  \min F_j$ and there exists $n^j \in F'_j$ such that $G_j \cap \mathrm{supp}f_{n^j}=\emptyset$. In particular, we have (\ref{res1}) and the rest of the proof is exactly the same as  in the successor case. 

\section{More about  the complex case}\label{more} In this section, we use \cite{KW}  to obtain, in the complex case, a characterization of the families $\F$ such that all isometries of $X_{\F,2}$ are signed permutations of the canonical basis (Proposition \ref{carac}), and a version of Theorem \ref{nonsense} where the assumption that the sequence space $\mathbf E$ is ``far from being Hilbertian'' has been removed (Proposition \ref{encore}).

\begin{proposition}\label{carac} Assume that $\K=\mathbb C$, and let $\mathcal F$ be a combinatorial family. The following are equivalent. 
\begin{enumerate}
\item[\rm (1)] Any isometry of $X_{\F,2}$ is a signed permutation of $(e_i)$.
\item[\rm (2)] For any $j\neq k$ in $\N$, 
one can find a set $F\in \mathcal F$ such that $F\cap \{ j,k\}\neq \emptyset$ and $\{ j,k\}\cup F\notin \F$.
\end{enumerate}
\end{proposition}

\begin{rem1} If the family $\F$ is compact, so that every set $F\in\F$ is contained in a set $G\in\F^{MAX}$, then, given $j,k\in\N$, the existence of a set $F$ as in (2) is equivalent to the existence of a set $G\in\F^{MAX}$such that $G$ contains $j$ or $k$ but not both. This property of the family $\F$ has been isolated by C. Brech and A. Tcaciuc \cite{AC}. As mentioned in the introduction, Proposition \ref{carac} is due to them, and it is valid in the real case as well. However, the proof we give here in the complex case looks rather different. 
\end{rem1}

\begin{rem2} For the sake of conciseness, given $j,k\in\N$, we will say that the family \emph{$\F$ satisfies Condition} $(*)_{j,k}$ if there exists a set $F\in \mathcal F$ such that $F\cap \{ j,k\}\neq \emptyset$ and $\{ j,k\}\cup F\notin \F$.
\end{rem2}

\smallskip For the proof of Proposition \ref{carac}, we need to recall some definitions from \cite{KW}. Let $Z$ be a complex Banach space. An operator $A$ on $Z$ is said to be \emph{hermitian} if $e^{it A}$ is an isometry for every $t\in\R$. A \emph{vector} $u\in Z$ is said to be hermitian if there is a hermitian projection $P$ on $Z$ such that ${\rm Ran}(P)=\mathbb C u$; and a \emph{linear subspace} $E\subset Z$ is hermitian if all vectors $u\in E$ are hermitian. Obviously, if $Z$ is a Hilbert space then $Z$ is a hermitian subspace of itself. The converse is true by \cite{KW}; in fact, any hermitian (closed) subspace of $Z$ is hilbertian.

\smallskip The following lemma is essentially contained in \cite{KW}, but we give a proof for completeness.
\begin{lemma}\label{hermi} Let $Z$ be complex Banach space with a normalized $1$-unconditional basis $(f_i)_{i\in\N}$, and let $j\neq k$ in $\N$. Let $Z_{j,k}:={\rm span}(f_j,f_k)$, and denote by $P_{j,k}$ the canonical projection of $Z$ onto $Z_{j,k}$. Also, let any $2\times 2$ matrix $M$ act on 
$Z_{j,k}$ by identifying $Z_{j,k}$ with $\mathbb C^2$ in the obvious way. The following are equivalent.
\begin{enumerate}
\item[\rm (i)] $Z_{j,k}$ is a hermitian subspace of $Z$.
\item[\rm (ii)] For any unitary $2\times 2$ matrix $U$, the operator $V:=UP_{j,k} +(I-P_{j,k})$ is an isometry of $Z$.
\end{enumerate}
\end{lemma}
\begin{proof} Assume that $Z_{j,k}\subseteq Z$ is hermitian. Then $Z_{j,k}$ is isometrically equal to $\ell_2(2)$ by~\cite{KW}, \textit{i.e.} $\Vert af_j+bf_k\Vert^2=\vert a\vert^2+\vert b\vert^2$ for any $a,b\in\mathbb C$. Let $U$ be a $2\times 2$ unitary matrix, and let us show that $V:=UP_{j,k} +(I-P_{j,k})$ is an isometry of $Z$. We may assume that $1$ is an eigenvalue of $U$ (otherwise multiply $U$ by a scalar), and that $U$ is not the identity; so it has a second eigenvalue $\omega\neq 1$. Let $(g,h)$ be an orthonormal basis of $Z_{j,k}$ with $Ug=\omega g$ and $Uh=h$. Then, with obvious notation and identifications:
\[ U= \omega P_g +(I_{j,k}-P_g),\]
where $P_g:Z_{j,k}\to Z_{j,k}$ is the orthogonal projection of $Z_{j,k}$ with range $\mathbb C g$.

Since $Z_{j,k}$ is hermitian, there is a hermitian projection $P$ of $Z$ such that ${\rm Ran}(P)=\mathbb C g$. Then $P$ must be $P_g P_{j,k}$. Indeed, denoting by $[\;,\;]$ the semi-scalar product used in \cite{KW}, for any $n\neq j,k$ we have 
$[g,f_n]=\langle f_n^*, g\rangle=0$, and hence $[f_n,g]=0$ by \cite[Proposition~3.3]{KW} because $f_n$ and $g$ are hermitian vectors. So $\langle g^*,f_n\rangle=0$, where $P={g^*\otimes g}$, {\it i.e.} $Pf_n=0$. Since this holds for any $n\neq j,k$, it follows that  $P=PP_{j,k}=P_gP_{j,k}$, the second equality being true because the restriction of $P$ to the Hilbert space $Z_{j,k}$ is a hermitian projection, \textit{i.e.} an orthogonal projection, and so it must be equal to $P_g$. 

Since $P$ is hermitian, $e^{itP}$ is an isometry for every $t\in\mathbb R$. But it is clear that $e^{itP}=V$ if $e^{it}=\omega$. Indeed, since $P^2=P$, we have
\[ e^{itP}= I+ (e^{it}-1) P= e^{it} P+(I-P).\] 
So, if $e^{it}=\omega$, then 
\[ e^{itP}=\omega P_gP_{j,k} + (I_{j,k}-P_g)P_{j,k}+ (I-P_{j,k})=UP_{j,k}+(I-P_{j,k})=V.\]

\smallskip
Conversely, assume that every operator $V$ as in (ii) is an isometry. Then, in particular, every $2\times 2$ unitary matrix $U$ defines an isometry of $Z_{j,k}$. It follows that $Z_{j,k}$ is isometrically equal to $\ell_2(2)$. Indeed, if $a,b\in\mathbb C$ are such that $\vert a\vert^2+\vert b\vert^2=1$, then one can find a unitary matrix $U$ such that $Uf_j=af_j+bf_k$, so $\Vert a f_j+b f_k\Vert =\Vert Uf_j\Vert =1$.

Now, let $g\in Z_{j,k}$; we must show that $\mathbb C g$ is the range of a hermitian projection $P$. We have no choice but setting $P:=P_gP_{j,k}$, where $P_g$ is the orthogonal projection of $Z_{j,k}=\ell_2(2)$ with range $\mathbb C g$. The same computation as above shows that if $t\in\mathbb R$, then $e^{itP}$ is the operator $V$ associated with $U:=e^{it} P_g+ (I_{j,k}-P_g)$; so $P$ is indeed a hermitian projection by (ii). 
\end{proof}

\smallskip In the next two lemmas, $\mathbf E$ is a complex Banach space with a normalized $1$-unconditional basis $(\mathbf e_i)$. If $j,k\in\N$, we set 
\[ \mathbf E_{j,k}:={\rm span}(\mathbf e_j,\mathbf e_k)\subset\mathbf{E}\qquad{\rm and}\qquad E_{j,k}:={\rm span}(e_j,e_k)\subset X_{\F,\mathbf E}.\]

\smallskip
\begin{lemma}\label{notisom} Assume that the basis $(\mathbf e_i)$ is strictly $1$-unconditional, and let $j\neq k$ in $\N$ be such that $\{ j,k\}\in\mathcal F$ and $\mathbf E_{j,k}$ is a hermitian subspace of $\mathbf E$. 
If $\F$ satisfies Condition~$(*)_{j,k}$, 
then $E_{j,k}$ is not a hermitian subspace of $X_{\F,\mathbf E}$.
\end{lemma} 
\begin{proof} Assume that $\F$ satisfies Condition $(*)_{j,k}$. By Lemma \ref{hermi}, we have to find a $2\times 2$ unitary matrix $U$ such that  the operator  $V:=UP_{j,k} +(I-P_{j,k})$ is not an isometry of $X_{\mathcal F,\mathbf E}$. We show that in fact, any unitary matrix $U$ with non-zero entries has this property.

\smallskip Write $U=\left(\begin{matrix} a&c\\
b& d\end{matrix}\right)$ with $a,b,c,d\neq 0$. Choose $F\in\F$ such that $F\cap\{ j,k\}\neq\emptyset$ and $\{ j,k\}\cup F\notin\F$. Assume for example that $j\in F$, and set $G:=F\setminus\{ j, k\}$. Then $\{j\}\cup G\in \F$ and $G\cup \{ j,k\}\notin\F$ (so $G\neq\emptyset$ since $\{ j,k\}\in\F$).

Let $x\in X_{\F,\mathbf E}$ have support equal to $G$. Then, since $\{ j\}\cup G\in\F$, 
\[ \Vert e_j+x\Vert =\Vert \mathbf e_j+\mathbf x\Vert_{\mathbf E}.\]
On the other hand,
\[ V(e_j+x)= ae_j+be_k+ x.\]
So, by strict $1$-unconditionality of $(\mathbf e_i)$ and since $\{ j,k\}\cup G\notin\F$ and $ab\neq 0$, we have 
\[ \Vert V(e_j+x)\Vert<\Vert a\mathbf e_j+b\mathbf e_k+\mathbf x\Vert_{\bf E}.\]
However, by Lemma \ref{hermi} we also have $\Vert a\mathbf e_j+b\mathbf e_k+\mathbf x\Vert_{\bf E}=\Vert \mathbf e_j+\mathbf x\Vert_{\mathbf E}$ because $\mathbf E_{j,k}$ is a hermitian subspace of $\mathbf E$ and  (with the obvious notation) 
$a\mathbf e_j+b\mathbf e_k+\mathbf x=\mathbf V(\mathbf e_j+\mathbf x)$. So $\Vert V(e_j+x)\Vert < \Vert e_j+x\Vert$, and hence $V$ is not an isometry. 
\end{proof}

\begin{lemma}\label{isom} Let $j\neq k$ in $\N$ 
be such that $\mathbf E_{j,k}$ is a hermitian subspace of $\mathbf E$. If $\F$ does not satisfy Condition $(*)_{j,k}$, 
then $E_{j,k}$ is a hermitian subspace of $X_{\F,\mathbf E}$.
\end{lemma}
\begin{proof} 
Assume that $\F$ does not satisfy Condition $(*)_{j,k}$. By Lemma \ref{hermi}, we have to show that for any unitary $2\times 2$ matrix $U$, the operator $V:=UP_{j,k} +(I-P_{j,k})$ is an isometry of $X_{\mathcal F,\mathbf E}$. It is in fact enough to show that $\Vert Vx\Vert\leqslant  \Vert x\Vert$ every $x\in X_{\F,\mathbf E}$, since we can then apply that to $V^{-1}$ (which has the same form as $V$); and we may assume that $x$ is finitely supported. 

\smallskip Choose $F\in\F$ such that $\Vert Vx\Vert=\Vert P_F\mathbf V\mathbf x\Vert_{\mathbf E}=\Vert P_FVx\Vert$. By assumption, we know that either $F\cap\{ j,k\}=\emptyset$ or $\{ j,k\}\cup F\in \F$. 
 
If  $F\cap\{ j,k\}=\emptyset$, then $P_F\mathbf V\mathbf x=P_F \mathbf x$, so $\Vert Vx\Vert=\Vert P_F\mathbf V\mathbf x\Vert_{\bf E} =  \Vert P_F\mathbf x\Vert_{\bf E} \leqslant  \Vert x\Vert$. 

Assume that $G:=\{ j,k\}\cup F\in\F$, so that $\Vert P_G x\Vert=\Vert P_G\mathbf x\Vert_{\mathbf E}$. Note that 
\[ P_G\mathbf V\mathbf x=\mathbf U P_{j,k} \mathbf x+P_G(I-P_{j,k})\mathbf x=\mathbf UP_{j,k} P_G\mathbf x+ (I-P_{j,k})P_G\mathbf x =\mathbf V P_G\mathbf x.\]
By Lemma \ref{hermi} and since $\mathbf E_{j,k}$ is a hermitian subspace of $\mathbf E$, it follows that $\Vert P_G\mathbf V\mathbf x\Vert_{\bf E}=\Vert P_G\mathbf x\Vert_{\mathbf E}$. So $\Vert Vx\Vert=\Vert P_F\mathbf V\mathbf x\Vert_{\bf E} \leqslant  \Vert P_G\mathbf V\mathbf x\Vert_{\bf E}=
\Vert P_G\mathbf x\Vert_{\mathbf E}\leqslant  \Vert x\Vert$.
\end{proof}

\smallskip
\begin{proof}[\it Proof of Proposition \ref{carac}] Assume (1), and let $j\neq k$ in $\N$. By (1) and Lemma \ref{hermi}, 
$E_{j,k}$ is not a hermitian subspace of $X_{\F,2}$; so $\F$ satisfies Condition $(*)_{j,k}$ by Lemma \ref{isom}. This shows that (1) implies (2). 
Conversely, assume (2). By \cite{KW}, to prove (1) it is enough to show that the Hilbert components of $X_{\F,2}$ are $1$-dimensional, \textit{i.e.} that $E_{j,k}$ is not a hermitian subspace of $X_{\F,2}$ for any $j\neq k$. If $\{ j,k\}\notin \F$ this is clear, since in that case $E_{j,k}$ is isometrically equal to $\ell_\infty(2)$ and so it is not hilbertian; and if $\{ j,k\}\in\F$, this follows from Lemma \ref{notisom}.
\end{proof}

\begin{corollary}\label{blablabla} Assume that $\K=\mathbb C$, and let $\F$ be a combinatorial family. Assume that for every $k\in\N$, one can find $F\in\F^{MAX}$ such that $\min F=k$. Then every isometry of $X_{\F,2}$ is an $\F$-$\,$compatible signed permutation of $(e_i)$.
\end{corollary}
\begin{proof} The assumption on $\F$ implies that for any $j<k$ in $\N$, one can find $F\in\F^{MAX}$ such that $k\in F$ but $j\notin F$; so the resut is clear by Proposition \ref{carac} and Proposition~\ref{trucmachin}~(1).
\end{proof}

\begin{remark} The assumption of the previous corollary is satisfied for example by any compact and spreading combinatorial family. It is also satisfied by some non-compact families, for example $\F:=\{ F\in{\rm FIN};\; F\subseteq 2\N\;{\rm or}\; F\subseteq \{ k,k+1\}\;\hbox{for some $k\in\N$}\}$.
\end{remark}

\smallskip We conclude this section with another result of the same kind. As always, $\mathbf E$ is a Banach space with a normalized $1$-unconditional basis $(\mathbf e_i)$.

\begin{proposition}\label{encore} Assume that $\K=\mathbb C$. Let $\F$ be a compact combinatorial family, and assume that {$\F$ satisfies Condition} $(*)_{j,k}$ for any $j\neq k$ in $\N$.  
 In each of the following two cases, every isometry of 
$X_{\F,\mathbf E}^*$ is an $\F$-$\,$compatible signed permutation of $(e_i^*)$.
\begin{enumerate}
\item[\rm (a)] $\Vert\,\cdot\,\Vert_\mathbf E$ is smooth at every finitely supported vector $\mathbf x\in S_{\mathbf E}$.
\item[\rm (b)] The basis $(\mathbf e_i)$ is strictly $1$-unconditional and, for every finite set $G\subset \N$, one can find an extreme point $\mathbf z^*$ of the unit ball of ${\rm span}(\mathbf e_i^*, i\in G)$ with support exactly equal to $G$.
\end{enumerate}
\end{proposition}
\begin{proof} We first note that since $\F$ is compact, $(e_i)$ is a shrinking basis of $X_{\F,\mathbf E}$ by Fact~\ref{domi}.  

\smallskip We also note that if an isometry of $X_{\F,\mathbf E}^*$ sign-permutes $(e_i^*)$, then the associated permutation of $\N$ must preserve $\F$: this follows from Remark 2 after Lemma \ref{ext} in Case~(a), and from Lemma \ref{extbis} in Case (b). So we just have to show that 
any isometry of $X_{\F,\mathbf E}^*$ sign-permutes $(e_i^*)$. Assume otherwise. Then, by \cite{KW} and since $(e_i^*)$ is a basis of $X_{\F,\mathbf E}^*$, 
there are $j\neq k$ in $\N$ such that ${\rm span}(e_j^*,e_k^*)$ is a hermitian subspace of $X_{\F,\mathbf E}^*$; 
and we must have $\{ j,k\}\in\F$ since otherwise ${\rm span}(e_j^*,e_k^*)$ is not hilbertian. So, for any unitary $2\times 2$ matrix $U$, the operator $V:= UP_{j,k}+(I-P_{j,k})$ is a well-defined isometry of $X_{\F,\mathbf E}^*$.  
A contradiction will arise from the following fact.

\begin{fact}\label{tiny} The support of any extreme point of $B_{X_{\F,\mathbf E}^*}$ belongs to $\F$; and conversely, for any set $G\in \F^{MAX}$, one can find $z^*\in {\rm Ext}\bigl( B_{X_{\F,\mathbf E}^*}\bigr)$ with support exactly equal to $G$.
\end{fact}
\begin{proof}[\it Proof of Fact \ref{tiny}] In Case (a), this is clear by Remark 2 after Lemma \ref{ext}. Assume we are in Case (b). By Lemma \ref{extbis},  any extreme point of $B_{X_{\F,\mathbf E}^*}$ has support in $\F$. Conversely, if $G\in\F^{MAX}$, one can find an extreme point $\mathbf z^*$ of the unit ball of 
${\rm span}(\mathbf e_i^*, i\in G)$ with support equal to $G$, and the associated vector $z^*\in {\rm span}(e_i^*, \, i\in G)$ is an extreme point of $B_{X_{\F,\mathbf E}^*}$ by Lemma \ref{extbis}.
\end{proof}

\smallskip Now, take a set $G\in\F^{MAX}$ such that (for example) $j\in G$ but $k\notin G$; so $G=\{ j\}\cup F$ where $F\cap\{ j,k\}=\emptyset$. By Fact \ref{tiny}, we may choose $z^*\in {\rm Ext}\bigl( B_{X_{\F,\mathbf E}^*}\bigr)$ with support exactly equal to $G=\{ j\}\cup F$. Then, for any $2\times 2$ unitary matrix $U$, the vector $Vz^* =UP_{j,k}z^*+(I-P_{j,k})z^*$ has support in $\F$, by Fact \ref{tiny} again. However, if $U$ has non-zero entries, then ${\rm supp}(Vz^*)=\{j,k\}\cup F=\{ k\}\cup G\notin\F$.
\end{proof}

\begin{remark} The proof has shown that in Case (b), it is enough to have the required property for every $G\in\F^{MAX}$. However, we wanted an assumption on $\mathbf E$ that does not depend on the family $\F$.
\end{remark}

\begin{corollary}\label{Scfin} Assume that $\K=\mathbb C$ and that either $\Vert \,\cdot\,\Vert_{\mathbf E}$ is smooth at every finitely supported $\mathbf x\in S_{\bf E}$, or $(\mathbf e_i)$ is strictly $1$-unconditional and, for every finite set $G\subset \N$, one can find an extreme point $\mathbf z^*$ of the unit ball of ${\rm span}(\mathbf e_i^*, i\in G)$ with support  equal to $G$. If  $(\Sc_\alpha)_{\alpha<\omega_1}$ is a transfinite Schreier sequence then, for any $1\leqslant\alpha <\omega_1$, every isometry of $X_{\Sc_\alpha, \mathbf E}^*$ is diagonal.
\end{corollary}
\begin{proof} Since $\Sc_\alpha$ is compact and spreading, the result follows from Proposition \ref{encore} together with the fact that the only permutation of $\N$ preserving $\Sc_\alpha$ is the identity.
\end{proof}


\begin{thebibliography}
\rm\bibitem{AK} F. Albiac and N. J. Kalton, \textit{Topics in Banach space theory.} Graduate Texts in Mathematics {\bf 233}, Springer (2006).
\rm\bibitem{AA} D. E. Alspach and S. A. Argyros, \emph{Complexity of weakly null sequences.}  Dissertationes Math. {\bf 321} (1992).
\rm\bibitem{AB} L. Antunes and K. Beanland, \emph{Surjective isometries on Banach sequence spaces: a survey.}  Concr. Oper. {\bf 9} (2022),  19--40.
\rm\bibitem{ABC} L. Antunes, K. Beanland and V. H. Chu, \emph{On the geometry of higher order Schreier spaces.} Illinois J. Math. {\bf 65} (2021), 47--69.
\rm\bibitem{BessPel} C. Bessaga and A. Pe\l czy\'nski, \emph{Spaces of continuous functions IV.} Studia Math. {\bf 19} (1960), 5--62.
\rm\bibitem{BiLau} A. Bird and N. Laustsen, \emph{An amalgamation of the Banach spaces associated with James and Schreier, Part I: Banach-space structure.} Banach algebras 2009, 45--76. 
Banach Center Publ., {\bf 91} (2010). 
\rm\bibitem{russian} M. S. Braverman and E. M. Semenov, \emph{Isometries of symmetric spaces}. Soviet Math. Doklady {\bf 15} (1974), 1027--1030.
\rm\bibitem{BFT} C. Brech, V. Ferenczi and A. Tcaciuc, \emph{Isometries of combinatorial Banach spaces.}  Proc. Amer. Math. Soc. {\bf 148} (2020), no. 11, 4845--4854. 
\rm\bibitem{BP} C. Brech and C. Pi$\tilde{\rm n}$a, \emph{Banach-Stone-like results for combinatorial Banach spaces.} Ann. Pure Appl. Logic {\bf 172} (2021), no. 8, Paper No. 102989.
\rm\bibitem{AC} C. Brech and A. Tcaciuc, \emph{Notes on isometries of $X_{\F}^p$.} Work in progress.
\rm\bibitem{Car} N. L. Carothers, \textit{A short course on Banach space theory.} London Mathematical Society Student Texts {\bf 64}. Cambridge University Press (2005).
\rm\bibitem{Cau} R. Causey, \emph{Concerning the Szlenk index.} Studia Math. {\bf 236} (2017), 201--244.
\rm\bibitem{FJ}  R. J. Fleming and J. E. Jamison, \textit{Isometries on Banach spaces: function spaces.} Chapman \& Hall/CRC Monographs and Surveys in Pure and Applied Mathematics {\bf 129}. 
Chapman \& Hall/CRC  (2003). 
\rm\bibitem{G} I. Gasparis, \emph{A dichotomy theorem for subsets of the power set of natural numbers}. Proc. Amer. Math. Soc. {\bf 129}, no. 3, 759--764.
\rm\bibitem{GL} I. Gasparis and D. H. Leung, \emph{On the complemented subspaces of the Schreier spaces.} Studia Math. {\bf 141} (2000), 273--300.
\rm\bibitem{Gr} R. Gra\'slewicz, \emph{Finite dimensional Orlicz spaces}.  Bull. Polish Acad. Sci. Math. {\bf 33} (1985),  277--283.
\rm\bibitem{J} R. C. James, \emph{Bases and reflexivity of Banach spaces}. Ann. Math {\bf 52} (1950), 518--527.
\rm\bibitem{KW} N. J. Kalton and G. V. Wood, \emph{Orthonormal systems in Banach spaces and their applications.} Math. Proc. Camb. Phil. Soc. {\bf 79} (1976), 493--510.
\rm\bibitem{LT} J. Lindenstrauss and L. Tzafriri, \textit{Classical Banach spaces I}. Springer (1977).
\rm\bibitem{Ph} R. R. Phelps,  \textit{Lectures on Choquet's theorem} (2nd edition). Lecture Notes in Mathematics {\bf 1757}. Springer (2001).
\rm\bibitem{Rain} J. Rainwater, \emph{Weak convergence of bounded sequences.} Proc. Amer. Math. Soc. {\bf 14} (1963), 999.
\rm\bibitem{Beata} B. Randrianantoanina, \emph{Injective isometries in Orlicz spaces}. Contemporary Math. {\bf 232} (1999), 269--287.
\rm\bibitem{RR} M. M. Rao and Z. D. Ren, \textit{Theory of Orlicz spaces}. Monographs and Textbooks in Pure and Applied Mathematics {\bf 146}, Marcel Dekker (1991).
\rm\bibitem{Ro} H. P. Rosenthal, \emph{The Banach space $\mathcal C(K)$}. Handbook of the Geometry of Banach spaces, vol. 2. North-Holland (2003).
\rm\bibitem{T} K. W. Tam, \emph{Isometries of certain function spaces.} Pacific J. Math {\bf 31} (1969), 233--246.
\end{thebibliography}
\end{document}